\documentclass[pdflatex,sn-mathphys-num]{sn-jnl}

\usepackage{graphicx}%
\usepackage{multirow}%
\usepackage{amsmath,amssymb,amsfonts}%
\usepackage{amsthm}%
\usepackage{mathrsfs}%
\usepackage[title]{appendix}%
\usepackage{xcolor}%
\usepackage{textcomp}%
\usepackage{manyfoot}%
\usepackage{booktabs}%
\usepackage{algorithm}%
\usepackage{algorithmicx}%
\usepackage{algpseudocode}%
\usepackage{listings}%

\theoremstyle{thmstyleone}%
\newtheorem{theorem}{Theorem}[section]
\newtheorem{proposition}[theorem]{Proposition}%
\newtheorem{lemma}[theorem]{Lemma}

\theoremstyle{thmstyletwo}%
\newtheorem{remark}{Remark}[section]%

\theoremstyle{thmstylethree}%
\newtheorem{definition}{Definition}[section]%
\newtheorem{problem}{Problem}[section]
\newtheorem{assumption}{Assumption}[section]
\newtheorem{property}{Property}[section]
\newtheorem{alg}{Algorithm}[section]

\begin{document} 

\title[Newton solver for parabolic PDEs with hysteresis]{Smoothing Newton Methods for Evolving Piecewise Smooth Systems Arising from Parabolic PDEs with Hysteresis}


\author*[1]{\fnm{Shu} \sur{Xu}}\email{mathxushu@pku.edu.cn}

\author[2]{\fnm{Liqun} \sur{Cao}}\email{clq@lsec.cc.ac.cn}


\affil*[1]{\orgdiv{School of Mathematical Sciences}, \orgname{Peking University}, \orgaddress{\city{Beijing}, \postcode{100871}, \country{People's Republic of China}}}


\affil[2]{\orgdiv{LSEC, NCMIS, Institute of Computational Mathematics and Scientiﬁc/Engineering Computing, Academy of Mathematics and Systems Science}, \orgname{Chinese Academy of Sciences}, \orgaddress{\city{Beijing}, \postcode{100190}, \country{People's Republic of China}}}


\abstract{
  Implicit discretizations of parabolic partial differential equations with hysteresis pose significant challenges for the design of robust nonlinear solvers due to the underlying history dependence.
  In this work, we derive a unified abstract formulation for the resulting nonlinear systems and show that their nonlinearities possess a monotone piecewise smooth structure with evolving switching behavior.
  Motivated by this structure, we develop a smoothing Newton method based on structure-preserving smoothing criteria, combining monotonicity preservation with an intermediate slope property.
  Under these conditions, the proposed method is shown to be globally convergent and locally Q-quadratically convergent.
  To enable practical implementation, we further introduce an explicit arc-based smoothing construction, thereby bridging the gap between abstract convergence theory and implementable algorithms.
  Numerical experiments for transient eddy-current problems with Preisach hysteresis demonstrate the robustness and efficiency of the proposed method compared with existing approaches.
}

\keywords{Hysteresis, Piecewise smooth nonlinear systems, Smoothing Newton method, Preisach model}

\pacs[MSC Classification]{47J40, 65H10, 65J15, 65M60}

\maketitle

\section{Introduction}\label{sec:intro}
  Hysteresis is a ubiquitous phenomenon across physics, chemistry, biology, and engineering \citep{bertotti_science_2006,nooriHysteresisPhenomenaBiology2014,morro_mathematical_2023}. Mathematically, it is characterized as a rate-independent memory effect \citep{visintinDifferentialModelsHysteresis1994,mielke_rate-independent_2015}. When hysteresis is coupled with diffusion or transport processes, it gives rise to parabolic partial differential equations with memory-dependent responses. Such models arise in a broad range of applications, including heat conduction with thermostatic control \citep{visintin_evolution_1986,tsuzuki_existence_2015,brokate_weak_2019} and simplified Maxwell's equations under the eddy-current approximation \citep{dupre_numerical_1996,vankeerNumericalMethod2D1996,bermudezElectromagneticComputationsPreisach2017,bermudezMathematicalAnalysisNumerical2020}.

  In this work, we consider nonlinear algebraic systems arising from implicit discretizations of such equations, including semilinear and quasilinear formulations of the form
  \begin{align}
  &\frac{\partial }{\partial t} u + \mathcal{A} u + \mathcal{W}(u, \xi^0) = f, \label{eq:preli_semi_pde_strong}\\
  &\frac{\partial }{\partial t} \left[u+\mathcal{W}\left(u, \xi^0\right)\right]  + \mathcal{A} u  = f,\label{eq:preli_quasi_pde_strong}
  \end{align}
  where $\mathcal{A}$ is a linear, bounded, self-adjoint, and coercive operator, $f$ is a given source term, $\xi^0$ denotes the initial memory configuration, and $\mathcal{W}$ is a space-distributed hysteresis operator. 
  A defining feature of hysteresis is its memory dependence: the value $\left[\mathcal{W}(u, \xi^0)\right](x,t)$ depends not only on the current input $u(x,t)$ but also on its past history $u(x,t')$, $0 \le t' < t$, and the initial memory state $\xi^0(x)$. For comprehensive theoretical foundations, we refer the reader to \cite{krasnoselskiiSystemsHysteresis1989,visintinDifferentialModelsHysteresis1994,krejčí1996hysteresis,brokateHysteresisPhaseTransitions1996}. 

  Numerical studies of parabolic equations with hysteresis remain an active area of research.  Implicit Euler finite element schemes are the standard approach for space-time discretization  \citep{verdiNumericalApproximationHysteresis1985,verdiNumericalApproximationPreisach1989,bermudez_mathematical_2014,bermudezElectromagneticComputationsPreisach2017,bermudez_mathematical_2020}, for which stability and convergence were established in the early works \citep{verdiNumericalApproximationHysteresis1985,verdiNumericalApproximationPreisach1989}.  Higher-order time-stepping schemes, such as the Crank-Nicolson method, have also been investigated \citep{vankeerNumericalMethod2D1996,vankeerComputationalMethodsEvaluation1998}, although they require modified temporal integration to handle non-monotone hysteretic evolution within each time step. To improve efficiency, various linearized schemes have been proposed  \citep{verdiNumericalApproximationPreisach1989,verdi_numerical_1994}, combining efficient linear solvers with pointwise nonlinear corrections. 

  Despite these advances, the evolving piecewise smooth structure of the resulting fully discrete systems remains a major obstacle to the development of efficient nonlinear solvers \citep{chiampi_accurate_1994,miklos_kuczmann_finite_2008,bermudezElectromagneticComputationsPreisach2017,peszynskaApproximationHysteresisFunctional2021}. 
  Most existing approaches rely on derivative-free iterations to avoid the difficulties associated with nonsmooth nonlinearities, including nonlinear SOR iterations \citep{verdiNumericalApproximationHysteresis1985,verdiNumericalApproximationPreisach1989}, fixed-point techniques \citep{vankeerNumericalMethod2D1996,vankeerComputationalMethodsEvaluation1998, bottauscioPreisachtypeHysteresisModels2000}, and duality methods \citep{bermudez_mathematical_2014,bermudezElectromagneticComputationsPreisach2017}. 
  Although these methods are often effective for isolated nonlinear systems, they typically suffer from slow convergence and strong sensitivity to algorithmic parameters.
  This issue becomes particularly pronounced in transient simulations, where each time step corresponds to a different nonlinear mapping whose smooth branches and  switching configurations may change considerably due to the history dependence of the hysteresis operator.
  As a result, parameter choices that perform well for certain time steps may lead to significantly deteriorated convergence for others.
  These observations naturally suggest the use of Newton-type methods that exploit the local smooth structure of the nonlinear systems.
  However, the intrinsic structure of such history-dependent nonsmooth systems has not yet been characterized in the existing literature.
  The first objective of this work is therefore to identify the evolving piecewise smooth structure and cast the fully discrete systems into a unified abstract formulation, thereby providing a mathematical foundation for the design and analysis of Newton-type methods.

  Recently, semismooth Newton methods \citep{hintermuller2010semismooth,ulbrich_semismooth_2011} have been applied to the numerical solution of PDEs with hysteresis \citep{bermudez_mathematical_2020,peszynskaApproximationScalarConservation2020,peszynskaApproximationHysteresisFunctional2021} and are known to achieve local quadratic convergence \citep{kojima_extension_1986,kummer_newtons_1992,qi_nonsmooth_1993}. 
  However, convergence failures have also been observed in practice \citep{peszynskaApproximationScalarConservation2020}, highlighting the need for effective globalization strategies.
  Since classical globalization theory for line-search Newton methods typically relies on smooth merit functions \citep{dennis_numerical_1996,de_luca_semismooth_1996}, its direct application to the present nonsmooth setting is nontrivial.
  Smoothing Newton methods \citep{chen_global_1998} provide a natural alternative by replacing the nonsmooth mapping with a family of smooth approximations.
  However, existing frameworks remain largely abstract and do not exploit the specific structural features of hysteresis-induced systems.
  To address this issue, we develop tailored smoothing criteria incorporating monotonicity preservation together with an intermediate slope property.
  Under these conditions, the resulting smoothing Newton method equipped with a backtracking line search is shown to be globally convergent and locally Q-quadratically convergent.
  
  Constructing smoothing approximations compatible with these criteria remains a nontrivial task for piecewise smooth functions \citep{arandiga_interpolation_2005,mazroui_simple_2005,lipman_approximating_2010,yilmaz_new_2019}. 
  The main difficulty lies in simultaneously preserving monotonicity and controlling the local derivative transition near switching points.
  Among existing approaches, the arc-based smoothing construction proposed in \cite{wu_smoothing_2015} is particularly attractive  due to its local geometric flexibility, although a complete theoretical justification together with a practical algorithmic realization remains unavailable.
  In this work, we rigorously justify this construction and develop a concrete algorithmic realization.
  Based on this approach, we explicitly construct a family of smoothing functions satisfying the proposed smoothing criteria.
  
  The main contributions of this paper are fourfold.
  First, we characterize the evolving piecewise smooth structure induced by hysteresis through level functions, and formalize a unified abstract algebraic model for the implicit Euler finite element discretizations of both \eqref{eq:preli_semi_pde_strong} and \eqref{eq:preli_quasi_pde_strong}, thereby providing a  foundation for the analysis and design of nonlinear solvers.
  Second, we propose structure-preserving smoothing criteria tailored to this class of problems, including monotonicity preservation and an intermediate slope property.
  Under these conditions, the associated smoothing Newton method equipped with a backtracking line search is proved to be globally convergent and locally Q-quadratically convergent.
  Third, we establish a rigorously justified arc-based smoothing construction together with a concrete algorithmic realization.
  The resulting approach yields explicit smoothing functions satisfying the proposed criteria, thereby bridging the gap between abstract smoothing theory and implementable nonlinear solvers.
  Finally, comprehensive numerical experiments are conducted for transient eddy-current problems with Preisach hysteresis.
  The test suite is designed to comprehensively evaluate convergence behavior, robustness across nonlinear systems evolving over time steps, and scalability under mesh refinement.
  The results demonstrate that the proposed method offers improved robustness and efficiency over existing approaches.
  
  The remainder of the paper is organized as follows.
  In Section~\ref{sec:preli}, we introduce the variational formulation of parabolic problems with hysteresis together with their fully discrete approximations.
  Section~\ref{sec:derivation_model} derives a unified abstract algebraic model and identifies the evolving piecewise smooth structure induced by hysteresis.
  In Section~\ref{sec:solver}, we propose structure-preserving smoothing criteria, establish global and local convergence results for the smoothing Newton method, and detail its practical implementation via an arc-based smoothing construction.
  Section~\ref{sec:experiments} reports numerical experiments for transient eddy-current problems with Preisach hysteresis, demonstrating the robustness, efficiency, and scalability of the proposed solver.
  Finally, Section~\ref{sec:conclusion} concludes the paper with several remarks and possible future directions.

\section{Parabolic Problems with Hysteresis}\label{sec:preli}
  This section recalls the mathematical setting of parabolic problems with hysteresis, together with the corresponding weak formulations and their numerical discretizations.
  In particular, we state several structural assumptions of continuous hysteresis operators that will play a key role in the characterization of the piecewise smooth structure.

  \subsection{Continuous Hysteresis Operators and Structural Assumptions}\label{subsec:preli_Op} 
    Continuous hysteresis refers to a type of hysteresis in which a continuous input generates a continuous output, thereby producing a continuous hysteresis loop.
    Formally, we consider a general operator
    \begin{equation}\label{eq:preli_Op_operator}
      \mathcal{F}\colon C[0, T]\times X \rightarrow C[0, T], \quad \left(u,\xi^0\right) \mapsto w=\mathcal{F}\left(u, \xi^0\right),
    \end{equation} 
    where $X$  is a metric space  and $\xi^0 \in X$ encodes the initial memory configuration. $\mathcal{F}$ is a hysteresis operator if it satisfies the following two fundamental properties:
    \begin{property}[Causality/Volterra]\label{asp:preli_Op_Causality}
      $\forall u_1, u_2 \in  C[0, T]$, $\xi^0\in X$, $t \in[0, T]$, if $u_1=u_2$ in $[0, t]$, then $[\mathcal{F}(u_1,\xi^0)](t)=[\mathcal{F}(u_2,\xi^0)](t)$.
    \end{property}
    \begin{property}[Rate independence]\label{asp:preli_Op_RateIndependence}
      $\forall\left(u, \xi^0\right) \in C[0,T]\times X$, if $s\colon [0, T] \rightarrow [0, T]$ is a continuous increasing function satisfying $s(0)=0$ and $s\left(T\right)=T$, then 
      \[
        \left[\mathcal{F}\left(u \circ s, \xi^0\right)\right](t)=\left[\mathcal{F}\left(u, \xi^0\right)\right](s(t)), \quad \forall t \in [0, T]. 
      \] 
    \end{property}

    Since the definition above is too general for structural analysis, additional assumptions are required.
    The following properties are standard in the theory of PDEs with hysteresis and will be used to derive the strongly monotone evolving piecewise smooth structure of the fully discrete nonlinear systems considered in this work. 

    \begin{assumption}[Piecewise monotonicity]\label{asp:preli_Op_pw_monotone}
      For all $\left(u, \xi^0\right) \in C[0, T]\times X$ and intervals $\left[t_1, t_2\right] \subset[ 0, T]$, if $u$ is either increasing or decreasing on $\left[t_1, t_2\right]$, then so is $\mathcal{F}\left(u, \xi^0\right)$.
    \end{assumption}
    \begin{assumption}[Piecewise Lipschitz continuity]\label{asp:preli_Op_pw_lipschitz}
      There exists $L>0$ such that for any $u \in C\left[0, T\right]$, $ \xi^0 \in X$ and interval $[t_1,t_2]\subset [0,T]$,  if $u$ is affine on $\left[t_1, t_2\right]$, then
      \[ 
        \left|[\mathcal{F}(u,\xi^0)]\left(t_1\right)-[\mathcal{F}(u,\xi^0)]\left(t_2\right)\right| \leq L\left|u\left(t_1\right)-u\left(t_2\right)\right|.
      \]
    \end{assumption}

    To account for spatially distributed hysteresis effects, we extend the operator $\mathcal{F}$ in \eqref{eq:preli_Op_operator} by introducing the space variable $x$ as a parameter:
    \begin{equation}\label{eq:preli_SpaceOp} 
      \left[\mathcal{W}\left(u,\xi^0\right)\right]\left(x, t\right):=\left[{\mathcal{F}}\left(u(x, \cdot), \xi^0(x)\right)\right](t), \quad \forall(x, t) \in \Omega \times[0, T],
    \end{equation}
    for $u : \Omega \times[0, T] \rightarrow \mathbb{R}$ and $\xi^0 :  \Omega  \rightarrow X$. 
    Thus, $\mathcal{F}$ acts pointwise in space, describing a local constitutive law for homogeneous materials. This extension introduces temporal memory effects but no spatial coupling.

    \begin{figure}[!htbp]
      \centering
      \includegraphics[width=0.46\textwidth]{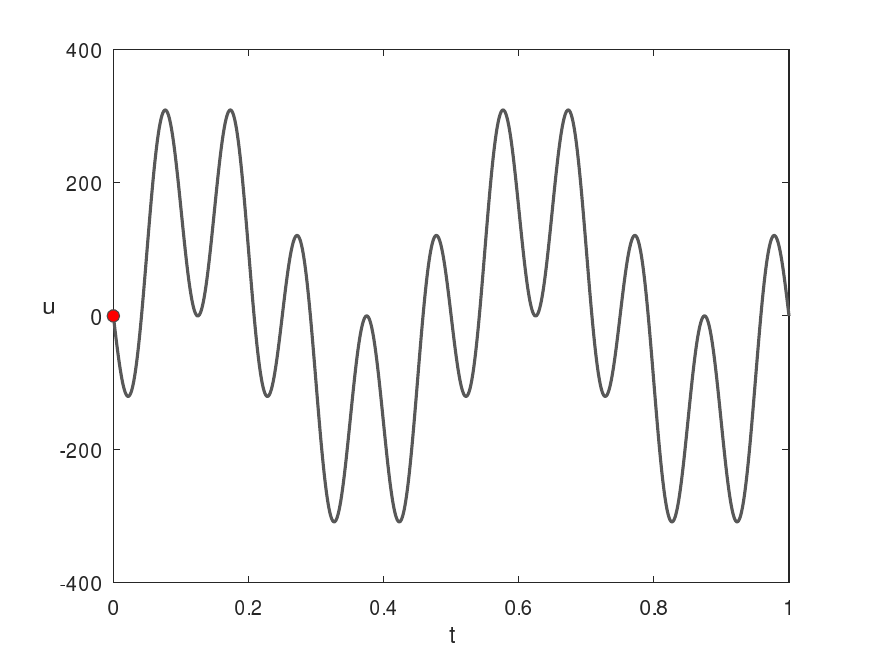}
      \put(-10,10){{(a)}}
      \hfill
      \includegraphics[width=0.46\textwidth]{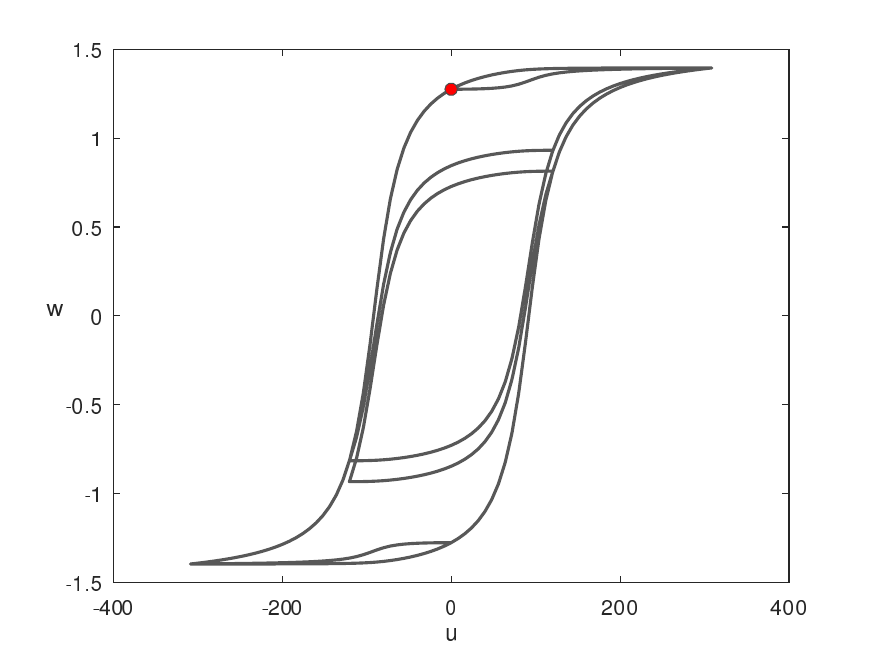}
      \put(-10,10){{(b)}}
      \caption{Illustration of the implemented Preisach hysteresis response using the benchmark excitation proposed in \cite{benabou_comparison_2003}. 
      (a) Excitation signal: $u(t) = 170\sin(4\pi t) + 170\sin(20\pi t + \pi)$. (b) Corresponding hysteresis input--output $u$-$w$ curve. }
      \label{fig:exp_preisach} 
    \end{figure}

    Among various hysteresis models, the Preisach operator is one of the most widely used models in electromagnetic applications due to its ability to accurately capture complex magnetic hysteresis phenomena.
    An illustration of the Preisach hysteresis response is shown in Figure~\ref{fig:exp_preisach}.
    Using the linear play operator $\mathcal{F}_{r}\colon u \mapsto w$ defined by  
    \begin{equation}\label{eq:preli_Op_linear_play}
      \begin{cases}
      u - {w} \in [-r,r], \quad \frac{\mathrm{d} w}{ \mathrm{d} t} (u - {w} - v) \geq 0,\quad \forall v\in [-r,r], \\ 
      w(0) = \max\left\{u(0)-r,  \min \left\{u(0)+r, 0\right\} \right\},
      \end{cases}
    \end{equation}
    the classical Preisach model \citep{mayergoyzMathematicalModelsHysteresis1991} can be written as  
    \begin{equation}\label{eq:def_preisach}
        \mathcal{P}[v](t)
        = 2 \int_{0}^{+\infty} \int_{0}^{\mathcal{F}_{r}[v](t)} \omega(r,\sigma)\, \mathrm{d}\sigma\, \mathrm{d}r,
    \end{equation}
    where $\omega\colon (0,+\infty)\times \mathbb{R} \rightarrow \mathbb{R}$ is the distribution function with  $\omega(r, \sigma) = \omega(r, -\sigma) \geq 0 $ and 
    \begin{equation*}
        \int_0^{+\infty} \sup_{\sigma \in \mathbb{R}} \omega(r, \sigma) \mathrm{d} r < \infty.
    \end{equation*}
    Under standard assumptions on the distribution function $\omega$, the Preisach operator $\mathcal{P}$ defines a continuous hysteresis operator satisfying Assumptions~\ref{asp:preli_Op_pw_monotone} and \ref{asp:preli_Op_pw_lipschitz}. 
    Hence, it naturally fits into the structural framework considered in this work.

  \subsection{Weak Formulation}

    Let $\Omega \subset \mathbb{R}^N$ be a bounded Lipschitz domain, and let $T>0$ denote the final time.
    We set $U=L^2(\Omega)$, with inner product $(\cdot,\cdot)$, and define
    \[
      V=H^1_{\Gamma_D}(\Omega)
      :=\{v\in H^1(\Omega):
      \operatorname{Tr}v=0\text{ on }\Gamma_D\},
    \]
    where $\Gamma_D$ is the Dirichlet boundary portion.
    The embedding $V\hookrightarrow U$ is continuous,
    dense and compact. Identifying $U$ with its dual
    via the Riesz representation theorem gives
    \[
      V\hookrightarrow U\cong U'\hookrightarrow V',
    \]
    where $V'$ denotes the continuous dual of $V$.

    Let $\mathcal A:V\to V'$ be a bounded, self-adjoint
    and coercive linear operator.
    Its associated bilinear form 
    \begin{equation*}
      a(u,v) =  {}_{V^{\prime}}\langle \mathcal{A} u, v\rangle_V 
    \end{equation*}
    is symmetric, and there exist positive constants $\alpha$ and $\beta$ such that 
    \begin{align*}
      \alpha \left\| u\right\|_{V}^2 \leq a(u,u) ,  \qquad 
      \left| a(u,v) \right| \leq \beta  \left\|u\right\|_V\left\|v\right\|_V,\qquad
      \forall u,v\in V.
    \end{align*}

    Let $f\in L^2(0,T;U)$ and $u^0\in V$.
    For a suitable initial memory configuration $\xi^0$
    and under appropriate additional mapping assumptions
    on the hysteresis operator
    (see, e.g.,
    \cite[Chapters~IX and~X]{visintinDifferentialModelsHysteresis1994}),
    we introduce the following weak formulations.
    These additional mapping assumptions are not needed
    for the nonlinear solver analysis below and are
    therefore not specified here.
    \begin{problem}\label{prob:preli_semi_pde_wp}
      Find $u\in H^1(0,T;U)\cap L^2(0,T;V)$ such that
      $w:=\mathcal W(u,\xi^0)\in C([0,T];U)$,
      $u(\cdot,0)=u^0$, and for almost every $t\in(0,T)$,
      \begin{equation}\label{eq:preli_semi_pde_weak}
          \left(\frac{\partial u}{\partial t}, \varphi \right) + a(u,\varphi) + \left(w , \varphi\right)   =\left( f ,\varphi\right),\quad \forall \varphi\in V.
      \end{equation}
    \end{problem}

    \begin{problem}\label{prob:preli_quasi_pde_wp}
      Find 
      $
      u \in H^1\left(0, T ; U\right)  \cap  L^2\left(0,T;V\right) 
      $
      such that $w:=\mathcal{W}(u, \xi^0) \in  H^1\left(0, T ; U\right)$,  $u(\cdot, 0) = u^0$, and for almost every $t\in(0,T)$,
      \begin{equation}\label{eq:preli_quasi_pde_weak}
        \left(\frac{\partial u}{\partial t}, \varphi \right) + \left(\frac{\partial w}{\partial t} , \varphi\right) + a(u,\varphi) = \left( f ,\varphi\right),\quad \forall \varphi\in V.
      \end{equation}
    \end{problem}

    \begin{remark}\label{rm:preli_axisymmetric}
      Well-posedness and regularity results for the
      continuous problems are available under additional
      assumptions; see
      \cite{visintinDifferentialModelsHysteresis1994,
      brokateHysteresisPhaseTransitions1996}.

      The framework also accommodates the weighted
      axisymmetric setting considered in
      \cite{bermudez_mathematical_2014,
      bermudezElectromagneticComputationsPreisach2017}.
      For an axisymmetric eddy-current problem posed
      on a meridian section $\Omega$ with coordinates
      $(r,z)$, the corresponding spaces are
      \[
        U=L^2_r(\Omega),\qquad
        V=\{v\in H^1_r(\Omega)\cap L^2_{1/r}(\Omega):
             \operatorname{Tr}v=0\text{ on }\partial\Omega\}.
      \]
      The bilinear form is replaced by its corresponding
      weighted counterpart.
    \end{remark}

  \subsection{Numerical Approximation}
    For the spatial discretization, assume that $\Omega$
    is polyhedral. Let $\{\mathcal T_h\}_h$ be a
    shape-regular family of simplicial meshes with
    mesh size $h$, resolving $\Gamma_D$.
    We use the conforming space $V_h\subset V$ of
    continuous piecewise linear functions vanishing
    on $\Gamma_D$ \citep{ciarlet2002finite}.
    Let $\{\psi_i\}_{i=1}^n$ be its nodal basis,
    associated with the free nodes $\{x_i\}_{i=1}^n$.
    To approximate the hysteresis terms, define the positive nodal weights
    \(
      d_i:=\int_\Omega\psi_i(x)\,\mathrm{d}x 
    \), $i=1,\ldots,n$. 
    For $z_h\in V_h$ and a function $v$ with specified
    nodal values,  we use the quadrature pairing
    \[
      (v,z_h)_h:=\sum_{i=1}^n d_i\,v(x_i)z_h(x_i).
    \]

    For the temporal discretization, let
    $t_k=k\tau$, $k=0,\ldots,K$, with the time step size $\tau=T/K$.
    Given a sequence $\{u_h^k\}_{k=0}^K\subset V_h$,
    let $u_{h\tau}$ be its continuous piecewise linear
    interpolant in time. Define
    \[
      w_{h\tau}:=\mathcal W(u_{h\tau},\xi^0),
      \qquad
      w_h^k:=w_{h\tau}(t_k),
      \quad k=0,\ldots,K.
    \]

    Let $u_h^0\in V_h$ approximate $u^0$, and set
    \(
      \bar f^k:=\frac1\tau
      \int_{t_{k-1}}^{t_k}f(s)\,\mathrm{d}s
    \).
    Combining the backward Euler method with the
    nodal quadrature for the hysteresis terms yields
    the following fully discrete formulations of
    \eqref{eq:preli_semi_pde_weak} and
    \eqref{eq:preli_quasi_pde_weak}.
    \begin{problem}\label{prob:semi_fem_full}
      Find $u_h^{k} \in  V_h$, $k=1,2,\dots,K$, such that for all $\varphi_h\in V_h$,
      \begin{equation}\label{eq:semi_fem_full}
        \left(\frac{u_h^{k} -  u_h^{k-1}}{\tau} , \varphi_h \right) + a(u_h^{k},\varphi_h) + \left( w_{h}^{k} , \varphi_h\right)_h   =\left(  \bar f^{k} ,\varphi_h\right).
      \end{equation}
    \end{problem}
    \begin{problem}\label{prob:quasi_fem_full}
      Find $u_h^{k} \in  V_h$, $k=1,2,\dots,K$, such that for all $\varphi_h\in V_h$,
      \begin{equation}\label{eq:quasi_fem_full}
        \left(\frac{u_h^{k} -  u_h^{k-1}}{\tau} , \varphi_h \right) + \left( \frac{ w_{h}^{k} - w_{h}^{k-1}}{\tau} ,\varphi_h\right)_h + a(u_h^{k},\varphi_h) = \left(  \bar f^{k} ,\varphi_h\right).
      \end{equation}
    \end{problem}
    \begin{remark}
      The causality of hysteresis operators ensures that the solution at time step $k$ depends only on the previous $u_h^j$, $j=0,1,\dots,k-1$. Consequently, the fully discrete problems define a sequence of nonlinear algebraic systems that can be solved sequentially in time.
      The well-posedness of Problems~\ref{prob:semi_fem_full} and~\ref{prob:quasi_fem_full} follows directly by the argument in~\cite{verdiNumericalApproximationHysteresis1985,verdiNumericalApproximationPreisach1989}.
    \end{remark}
    \begin{remark}
      In the axisymmetric setting of Remark~\ref{rm:preli_axisymmetric}, continuous piecewise linear functions vanishing on the symmetry axis belong to
      $H^1_r(\Omega)\cap L^2_{1/r}(\Omega)$; see \cite[Section~7]{bermudez_mathematical_2014}.
      Imposing the homogeneous Dirichlet condition on $\partial\Omega$ gives $V_h\subset V$.
      The corresponding nodal weights are
      \(
        d_i=r_i\int_\Omega\psi_i(x)\,\mathrm{d}x
      \),
      where $r_i$ is the radial coordinate of $x_i$.
      Since the free nodes lie off the symmetry axis,
      these weights remain positive.
    \end{remark}

\section{Evolving Piecewise Smooth Structure}\label{sec:derivation_model}
The purpose of this section is to identify the evolving piecewise smooth structure underlying the nonlinear algebraic systems generated by Problems~\ref{prob:semi_fem_full} and \ref{prob:quasi_fem_full}.
To expose the essential structure, we introduce the following unified abstract algebraic problem.
\begin{problem}\label{prob:solver_model_full}
  Given a symmetric positive-definite matrix $A \in \mathbb{R}^{n\times n} $, find $u^k \in \mathbb{R}^n$, $k=1,\dots,K$, such that 
  \begin{equation}
    H^k(u^k) := Au^k + F^k(u^k) - f^k = 0.
    \label{eq:solver_model_full}
  \end{equation}   
  Here, the vector \(f^k\in\mathbb R^n\) and the nonlinear mapping \(F^k:\mathbb R^n\to\mathbb R^n\)
  are determined by the previously computed $u^0,\dots, u^{k-1}$. Moreover, $F^k$ is a piecewise $C^{2}$-function\footnote{Here, a piecewise \(C^{k}\)-function is understood to consist of only finitely many \(C^{k}\)-smooth pieces and therefore forms a strict subclass of the \(PC^k\) functions in \cite{scholtes_introduction_2012}, the latter allowing for infinitely many such pieces.}  whose component functions satisfy
  \[
    F_i^k(x)=\phi_i^k(x_i), \qquad i=1,\dots,n,
  \]
  where each $\phi_i^k$ is increasing and has derivative bounded almost everywhere.
\end{problem}
  \begin{remark}
    Problem~\ref{prob:solver_model_full} represents a family of nonlinear systems of the generic form
    \begin{equation}
      H(u) := Au + F(u) - f = 0,
      \label{eq:solver_model}
    \end{equation}
    where \(A\) is symmetric positive definite and \(F\) is componentwise increasing with   \(
      F_i(x)=\phi_i(x_i), \, i=1,\dots,n
    \).
    Consequently, \(H\) is strongly monotone, since
    \begin{equation}\label{eq:solver_strong_monotone}
        \left(H(u_1)-H(u_2),u_1-u_2\right)
        \geq
        \lambda_{\min}(A)\|u_1-u_2\|^2,
    \end{equation}
    where \(\lambda_{\min}(A)\) denotes the smallest eigenvalue of \(A\).
    Therefore, each nonlinear system in Problem~\ref{prob:solver_model_full}
    admits a unique solution by \cite[Cor.~10.42]{renardy_introduction_2004}.
  \end{remark}

  The principal difficulty of Problem~\ref{prob:solver_model_full}
  arises from  the evolving piecewise smooth structure of the nonlinear mappings \(F^k\).
  Although the nonlinear systems share a similar algebraic form,
  their piecewise smooth structure may vary substantially across the solution sequence due to hysteresis history evolution.
  This observation motivates the structure-preserving smoothing Newton method with globalization developed in the next section.

  \subsection{Level-Function Characterization}
    To characterize the piecewise smooth structure, it is necessary to isolate the local scalar response generated by the hysteresis operator along a prescribed history path.
    This can be achieved through the level-function formalization introduced in \cite{brokateHysteresisPhaseTransitions1996}, which transforms the history-dependent hysteresis response into a family of scalar nonlinear mappings.

    Recall that $u_{h\tau}$ is the continuous piecewise linear interpolant in time of  $\{u_h^k\}_{k=0}^K$. 
    Since $w_{h\tau}(x) = \mathcal{F}\left(u_{h\tau}(x,\cdot),\xi^0(x)\right)$, we restrict the hysteresis operator $\mathcal{F}$ on the space $C_\mathrm{pm}\left[0,T\right] \times X$, where
    \begin{equation*}
      C_\mathrm{pm}\left[0,T\right] = \left\{ v\in C\left[0,T\right]\colon \text{ $v$ is piecewise monotone}  \right\}.
    \end{equation*}
    To formalize this restriction, we encode discrete temporal histories through finite sequences.
    Let $S$ denote the set of all finite sequences of real numbers:
    \[
      S=\left\{\left(v_0, v_1, \ldots, v_N\right) \mid N \in \mathbb{N}_0, v_i \in \mathbb{R}, 0 \leq i \leq N\right\},
    \]
    and define the prolongation operator $\pi\colon S\rightarrow C_\mathrm{pm}\left[0,T\right] $, $\left(v_0, v_1, \ldots, v_N\right) \mapsto v$, where 
    \[
      v\left((k+\alpha) \frac{T}{N}\right):= \alpha v_{k+1} + (1-\alpha) v_k, \quad \alpha\in \left[0,1\right], \quad 0 \leq k \leq N-1.
    \]
    For the restriction $\left. \mathcal{F} \right|_{C_\mathrm{pm}\left[0,T\right]\times X}$ we define the associated level functions $l_N: \mathbb{R} \rightarrow \mathbb{R}$,
    \begin{equation}\label{eq:solver_level_func}
      l_N(x)=\left[\mathcal{F}\left( \pi \left(v_0, \ldots, v_{N-1}, x\right), \xi^0\right)\right]\left(T\right)
    \end{equation}
    for fixed $\xi^0\in X$, $N \in \mathbb{N}_0$ and $(v_0, \dots, v_{N-1}) \in S$.  For the history path $(v_0,\dots,v_{N-1}, x)$, the level function describes the terminal hysteresis response generated by varying only the last value $x$.

    \begin{proposition}\label{prop:level_func_equiv}
      If a continuous hysteresis operator $\mathcal{F}$ is piecewise monotone (resp. Lipschitz continuous), the level functions of $\left. \mathcal{F} \right|_{C_\mathrm{pm}\left[0,T\right]\times X}$ are increasing (resp. Lipschitz continuous).  
    \end{proposition}
    \begin{proof}
      Fixing $\xi^0 \in X$, $N \in \mathbb{N}_0$, and $(v_0, \dots, v_{N-1}) \in S$, the level function $l_N$ is defined by \eqref{eq:solver_level_func}.  
      Without loss of generality, let $x_1, x_2$ satisfy either $x_2 > x_1 \geq v_{N-1}$ or $x_2 < x_1 \leq v_{N-1}$.  
      For $i=1,2$, define
      \[
      u_i := \pi(v_0, \dots, v_{N-1}, x_i),
      \]
      \[
      u_{1i} := \pi(v_0, \dots, v_{N-1}, x_1, x_i),
      \]
      and continuous increasing functions $s_i : [0,T] \to [0,T]$ with $s_i(0)=0$ and $s_i(T)=T$ by
      \begin{equation*}
      s_1(t) =
      \begin{cases}
      \frac{N+1}{N}\, t, & t \in \left[0, \frac{NT}{N+1} \right], \\[6pt] 
      T, & t \in \left[\frac{NT}{N+1}, T \right],
      \end{cases}
      \end{equation*}
      and
      \begin{equation*}
      s_2(t) =
      \begin{cases}
      \frac{N+1}{N}\, t, & t \in \left[0, \frac{N-1}{N+1}T \right], \\[6pt]
      T' + \frac{(N+1)\!\left(T' - \frac{N-1}{N}T\right)}{T}
      \left(t - \frac{N}{N+1}T\right), 
      & t \in \left[\frac{N-1}{N+1}T,\, \frac{N}{N+1}T \right], \\[8pt]
      T' + \frac{(N+1)(T - T')}{T}
      \left(t - \frac{N}{N+1}T\right), 
      & t \in \left[\frac{N}{N+1}T,\, T \right],
      \end{cases}
      \end{equation*}
      where
      \[
      T' := T - \frac{x_2 - x_1}{x_2 - v_{N-1}} \cdot \frac{T}{N}.
      \]
      By construction, we have
      \[
      u_1 \circ s_1 = u_{11}, 
      \qquad 
      u_2 \circ s_2 = u_{12}.
      \]
      Noticing $s_1\left(\tfrac{N}{N+1}T\right) = s_2(T) = T$, it follows by \eqref{eq:solver_level_func} and Property~\ref{asp:preli_Op_RateIndependence} that 
      \begin{equation}\label{eq:rate_ind_flat_equal}
      l_N(x_1)
      = \bigl[\mathcal{F}(u_1,\xi^0)\bigr]\!\left(s_1\!\left(\tfrac{N}{N+1}T\right)\right)
      = \bigl[\mathcal{F}(u_{11},\xi^0)\bigr]\!\left(\tfrac{N}{N+1}T\right),
      \end{equation}
      and
      \begin{equation*}
      l_N(x_2)
      = \bigl[\mathcal{F}(u_2,\xi^0)\bigr]\!\left(s_2(T)\right)
      = \bigl[\mathcal{F}(u_{12},\xi^0)\bigr](T).
      \end{equation*}
      Since
      \[
      \left. u_{11} \right|_{\left[0, \frac{N}{N+1}T \right]}
      =
      \left. u_{12} \right|_{\left[0, \frac{N}{N+1}T \right]},
      \]
      Property~\ref{asp:preli_Op_Causality} implies
      \begin{equation*}
      l_N(x_1)
      = \bigl[\mathcal{F}(u_{11},\xi^0)\bigr]\!\left(\tfrac{N}{N+1}T\right)
      = \bigl[\mathcal{F}(u_{12},\xi^0)\bigr]\!\left(\tfrac{N}{N+1}T\right).
      \end{equation*}
      Since $u_{12}$ is affine on $\left[\tfrac{N}{N+1}T,\, T\right]$, Assumption~\ref{asp:preli_Op_pw_monotone} yields
      \begin{equation*}
      \bigl(l_N(x_1) - l_N(x_2)\bigr)\,(x_1 - x_2) \ge 0,
      \end{equation*}
      and Assumption~\ref{asp:preli_Op_pw_lipschitz} gives
      \[
      \left| l_N(x_1) - l_N(x_2) \right|
      \le L \left| u_{12}\!\left(\tfrac{N}{N+1}T\right) - u_{12}(T) \right|
      = L |x_1 - x_2|.
      \]
      \end{proof}
    At this stage, Assumptions~\ref{asp:preli_Op_pw_monotone} and \ref{asp:preli_Op_pw_lipschitz} can be conveniently formalized in terms of level functions. In addition, we need the following piecewise $C^2$-continuity assumption, which is natural in practice. 
    \begin{assumption}[Piecewise $C^{2}$-continuity\footnote{This notion originates from \cite[Def.~2.2.14]{brokateHysteresisPhaseTransitions1996}.}]\label{asp:preli_Op_pw_C2}
      The level functions $l_N: \mathbb{R} \rightarrow \mathbb{R}$ defined by \eqref{eq:solver_level_func} for $\left. \mathcal{F} \right|_{C_\mathrm{pm}\left[0,T\right]\times X}$ are 
      piecewise $C^2$ for any $\xi^0\in X$, $N \in \mathbb{N}_0$ and $\left(v_0, \ldots, v_{N-1}\right) \in S$.
    \end{assumption}

    The level-function representation reveals that terminal hysteresis responses consist of monotone smooth branches separated by switching points.
    Moreover, as the finite sequence evolves, both the smooth branches and the associated switching locations may change throughout the transient process (see Figure~\ref{fig:exp_evolving_structure}). 
    This evolving piecewise smooth structure forms the fundamental source of the nonlinear difficulties.

  \subsection{Structure of the Fully Discrete Systems}
  We now show that the fully discrete formulations
  \eqref{eq:semi_fem_full} and \eqref{eq:quasi_fem_full}
  naturally give rise to nonlinear systems of the abstract form
  \eqref{eq:solver_model_full}, where the nonsmooth nonlinearities remain
  localized at the nodal level.

  Using the nodal basis introduced above, define
  the mass, lumped-mass and stiffness matrices by
    \begin{equation*}
        M_{ij} =  \left( \psi_{j}, \psi_{i} \right),\quad D_{ij} = \delta_{ij} d_i, \quad K_{ij} =a \left(\psi_{j}, \psi_i \right). 
    \end{equation*}
    Define $A = M + \tau K \in \mathbb{R}^{n\times n}$, $f_i=(\tau\bar f^k+u_h^{k-1},\psi_i)$, let $u \in \mathbb{R}^n$ be the coefficient vector of $u_h^{k} \in V_h$, and set
    \begin{align}
      &\left(F^{k}(u)\right)_i = 
        \begin{cases}
          \tau w_{h}^{k}\left( x_i\right) D_{ii}  &\text{ for } \eqref{eq:semi_fem_full},\\
          \left(w_{h}^{k}-  w_{h}^{k-1}\right)\left( x_i\right) D_{ii} &\text{ for } \eqref{eq:quasi_fem_full}.\\
        \end{cases}  \label{eq:solver_def_F}
    \end{align}
    Due to the pointwise action of the hysteresis operator
    \eqref{eq:preli_SpaceOp} together with the nodal quadrature,
    the nonlinear contribution at each component depends only on the nodal value
    at the same spatial location.
    Then \eqref{eq:semi_fem_full} and \eqref{eq:quasi_fem_full} can be rewritten in the unified form \eqref{eq:solver_model_full}.

    \begin{proposition}\label{prop:solver_nonlinearity} 
      Let $F$ be defined by \eqref{eq:solver_def_F}. Then its $i$-th component depends only on the $i$-th argument, i.e., $(F(u))_i = \phi_i(u_i)$.  
      Suppose that Properties~\ref{asp:preli_Op_Causality}
      and~\ref{asp:preli_Op_RateIndependence},
      together with Assumptions~\ref{asp:preli_Op_pw_monotone},
      \ref{asp:preli_Op_pw_lipschitz},
      and~\ref{asp:preli_Op_pw_C2}, hold.
      Then  each $\phi_i$ is an increasing piecewise $C^2$-function with derivative bounded almost everywhere.
    \end{proposition}
    \begin{proof}
      Since $u_{h\tau}$ is affine on $\left[t_{k-1},t_{k}\right]$ and $\left.u_{h\tau}\right|_{\left[0,t_{k-1}\right]}$ is known, for the node $x_i$ setting 
      \[
        v_{x_i} = \mathcal{F}\left(\pi \left(u_h^0(x_i),\dots, u_h^{k-1}(x_i), \underbrace{ u_h^{k}(x_i), \dots, u_h^{k}(x_i)}_{K-k+1}\right),\xi^0(x_i)\right) \in C[0,T],
      \]
      Property~\ref{asp:preli_Op_Causality} gives
      \(
        w_{h}^{k}\left( x_i\right) =  v_{x_i} (t_{k}).
      \)
      Notice $ u_h^k(x_i)=u_i$ and hence $(F(u))_i = \phi_i(u_i)$ by definition \eqref{eq:solver_def_F}. 
      In a similar manner to \eqref{eq:rate_ind_flat_equal}, Property~\ref{asp:preli_Op_RateIndependence} yields
      \begin{equation*}
          v_{x_i} (t_{k}) =  v_{x_i} (T)  
          =  \mathcal{F}\left(\pi \left(u_h^0(x_i),\dots, u_h^{k-1}(x_i),  u_h^{k}(x_i)\right),\xi^0(x_i)\right) (T) =: l_{k}(u_i),
      \end{equation*}
      where $l_{k}$ is associated with $\xi^0(x_i)\in X$ and $(u_h^0(x_i),\dots, u_h^{k-1}(x_i)) \in S$ . 
      Hence, by Proposition~\ref{prop:level_func_equiv} and Assumption~\ref{asp:preli_Op_pw_C2}, each $\phi_i$ is increasing, piecewise $C^2$ and  Lipschitz continuous.
      Therefore, $\phi_i'$ is bounded almost everywhere.  
    \end{proof}

    \subsection{Illustration of the Evolving Structure}
    
      To further illustrate the evolving piecewise smooth structure generated by hysteresis evolution, we consider again the Preisach benchmark introduced in Figure~\ref{fig:exp_preisach}.  
      
      \begin{figure}[!htbp]
        \centering
        \includegraphics[width=0.7\textwidth]{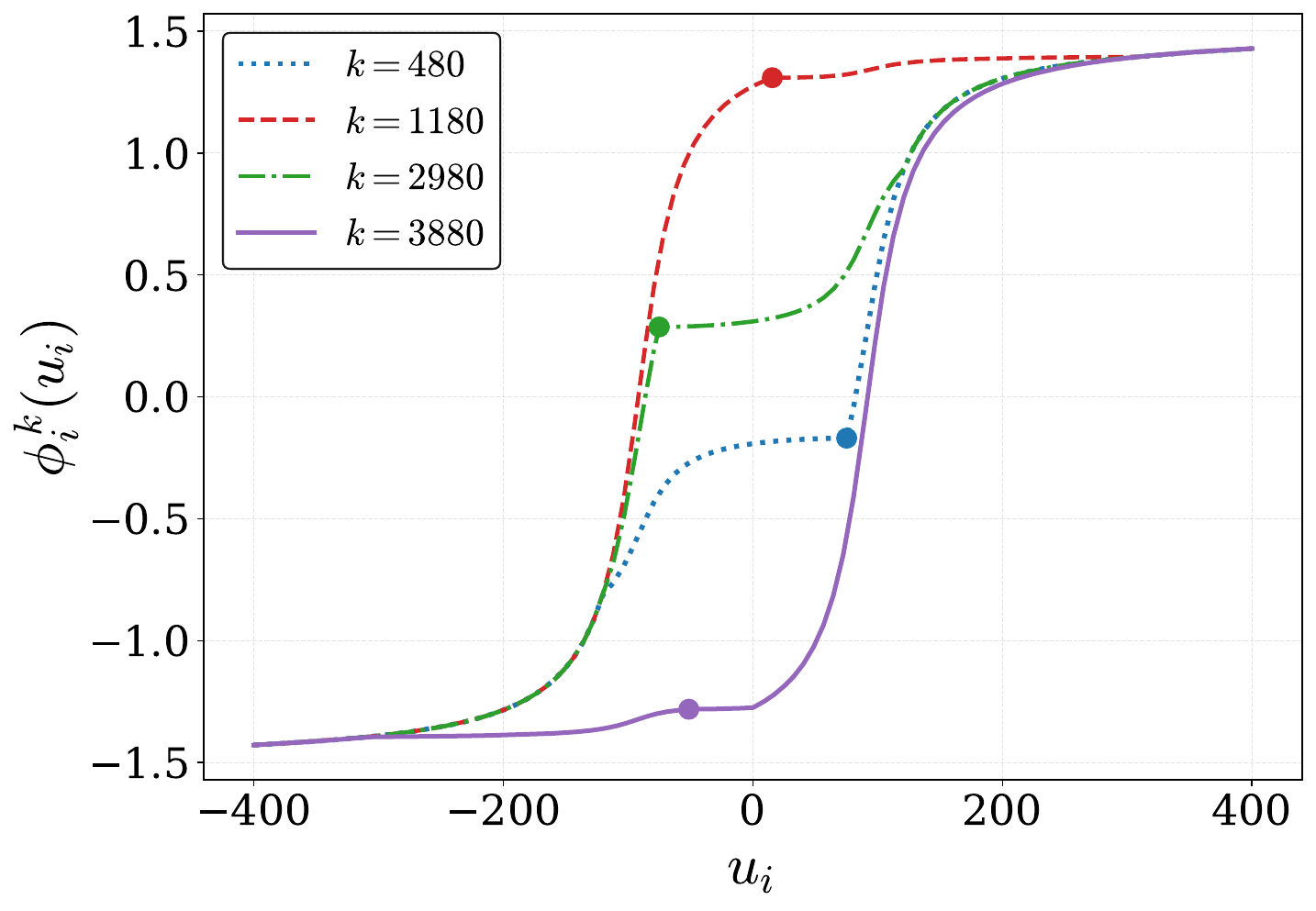}
        \caption{
          Representative local nonlinear functions $\phi_i^k$ generated during the transient evolution of the Preisach benchmark problem with $K=10000$ time steps.
        }
        \label{fig:exp_evolving_structure}
      \end{figure}
 
      Figure~\ref{fig:exp_evolving_structure} shows representative nonlinear functions corresponding to the same spatial degree of freedom evaluated at different steps of the transient evolution. 
      Each curve represents the scalar component function $\phi_i^k$ appearing in 
      \(
      (F^k(u))_i = \phi_i^k(u_i)
      \)
      for a fixed $i$,
      as characterized in Proposition~\ref{prop:solver_nonlinearity}. 
      The figure illustrates that the nonlinear functions remain monotone and piecewise smooth, while the locations of switching points and smooth branches evolve throughout the transient process.
      This observation highlights the need for robust Newton-type methods capable of handling changing piecewise smooth structures throughout the transient process.

\section{Jacobian Smoothing Newton Methods}\label{sec:solver}

  Building on the monotone piecewise smooth structure identified in the previous section, we develop a structure-preserving Jacobian smoothing Newton method incorporating a backtracking line-search globalization strategy. 
  To this end, we first introduce a set of smoothing criteria ensuring global convergence and local Q-quadratic convergence of the resulting method. 
  We then present a concrete arc-based smoothing construction together with a practical algorithmic realization satisfying these criteria.

  \subsection{From Classical Newton to Jacobian Smoothing}
    To motivate the smoothing method, we begin with the classical Newton method applied to \eqref{eq:solver_model}.
    On each smooth branch, the Jacobian of $H$ takes the diagonal-perturbed form
    \begin{equation} \label{eq:newton_jacobian}
    J = A +  \sum_{i=1}^n \phi_i' e_i e_i^T,  
    \end{equation}
    whose sparsity pattern coincides with that of $A$. However, due to the piecewise structure, the derivatives $\phi_i'$ are generally discontinuous and undefined at switching points. 
    Moreover, the natural merit function associated with line-search globalization,
    \begin{equation}\label{eq:solver_merit}
      \theta(x) = \frac{1}{2}\left\| H(x)\right\|^2,
    \end{equation}
    is itself generally nonsmooth.  As a consequence, the standard globalization theory for damped Newton methods based on smooth merit functions does not apply directly to Problem~\ref{prob:solver_model_full}, and global convergence cannot in general be guaranteed.

    Since each component function $\phi_i$ of $F$ is piecewise $C^1$, for any $x_i \in \mathbb{R} $, there exist $\delta > 0$ and two $C^1$ functions $\phi_{i,1}, \phi_{i,2} : (x_i - \delta, x_i + \delta) \to \mathbb{R}$ such that 
    \[
    \phi_i = \phi_{i,1} \text{ on } (x_i - \delta, x_i], 
    \qquad 
    \phi_i = \phi_{i,2} \text{ on } [x_i, x_i + \delta).
    \]
    Hence, the B-subdifferential $\partial_B \phi_i$ at each point $x_i$ can be written explicitly as
    \begin{equation}\label{eq:solver_component_Bdiff}
      \partial_B \phi_i(x_i) =
        \{\phi_{i,1}'(x_i),\, \phi_{i,2}'(x_i)\}.
    \end{equation}
    Consequently, the generalized derivative structure of $F$ can be characterized explicitly in terms of one-sided derivatives. In particular, both the Clarke generalized Jacobian and Qi's $C$-subdifferential admit simple componentwise representations
    \begin{gather*}
      \partial F(x) = \operatorname{co}\!\big(\partial_B \phi_1(x_1) \times \cdots \times \partial_B \phi_n(x_n)\big), \\
      \partial_C F(x) = \operatorname{co}\!\big(\partial_B \phi_1(x_1)\big) \times \cdots \times \operatorname{co}\!\big(\partial_B \phi_n(x_n)\big).
    \end{gather*}
    Since piecewise $C^2$ functions are strongly semismooth
    \citep{scholtes_introduction_2012,ulbrich_semismooth_2011},
    the nonlinear mapping $H$ is strongly semismooth as well.
    Therefore, Problem~\ref{prob:solver_model_full} can be solved by a semismooth Newton method with local quadratic convergence \citep{kojima_extension_1986}.
    However, global convergence remains delicate.
    The damped semismooth Newton method~\citep{de_luca_semismooth_1996} achieves globalization by employing a smooth merit function in the nonlinear complementarity setting, ensuring convergence of line search steps.  
    In contrast, in our context the merit function~\eqref{eq:solver_merit} is still nonsmooth.

    To avoid these difficulties, we employ a Jacobian smoothing strategy.
    The central idea of Jacobian smoothing Newton methods is to replace the nonsmooth mapping $F$ by a family of smooth approximations
    \begin{equation*}
      \tilde{F}(x,\varepsilon) : \mathbb{R}^n \times (0,+\infty) \longrightarrow \mathbb{R}^n.
    \end{equation*}
    The main challenge lies in designing such approximations so as to simultaneously ensure globalization and local Q-quadratic convergence. We next recall a generic Jacobian smoothing Newton framework that will serve as the basis for our subsequent structure-preserving construction.

    \begin{alg}[Generic Jacobian Smoothing Newton Framework \citep{chen_global_1998}]\label{alg: jacobian_smoothing} Let $\tilde{H}(x,\varepsilon) := Ax +\tilde{F}(x,\varepsilon) -f$.
      Given constants $\rho,\ \alpha,\ \eta\in(0,1)$, $\gamma\in(0,+\infty)$, and an initial guess $x^0\in \mathbb{R}^n$, choose $\sigma \in (0,\frac{1}{2}(1-\alpha))$ and $\mu >0$ such that for all $ x\in \mathbb{R}^n$ and $\varepsilon>0$,
      \begin{equation*}
        \left\| \tilde H(x,\varepsilon)-H(x)\right\|\le \mu\varepsilon.
      \end{equation*}
      Set $\beta^0=\left\| H(x^0)\right\|$, $\varepsilon^0=\frac{\alpha}{2\mu}\beta^0$ and $k=0$. Then iterate as follows:
  
      (1) Compute the Newton direction $d^k$ by solving
      \begin{equation}\label{eq: smoothing_linearization_problem}
      H(x^k) + \tilde H'_x(x^k,\varepsilon^k)d^k=0,
      \end{equation}
      where $\tilde H'_x$ denotes the Jacobian of $\tilde H$ with respect to $x$.
  
      (2) Let $m_k$ be the smallest nonnegative integer $m$ such that
      \begin{equation}\label{neq: linesearch}
      \theta_k(x^k+\rho^md^k) - \theta_k(x^k) \le -2\sigma\rho^m\theta(x^k),
      \end{equation}
      where $\theta_k(x) = \frac{1}{2}\left\| \tilde{H}(x,\varepsilon^k)\right\|^2$.
      Set $t_k=\rho^{m_k}$ and $x^{k+1}=x^k+t_k d^k$.
  
      (3) If $H(x^{k+1})=0$, terminate. Otherwise, if
      \begin{equation}\label{neq: decrease}
      \left\| H(x^{k+1})\right\|\le \max\left\{\eta\beta^k, \alpha^{-1}\left\| H(x^{k+1}) - \tilde H(x^{k+1},\varepsilon^k)\right\|\right\},
      \end{equation}
      set $\beta^{k+1}=\left\| H(x^{k+1})\right\|$ and choose $\varepsilon^{k+1}$  satisfying 
      \begin{equation*}
        0<\varepsilon^{k+1}\le \min\left\{\frac{\alpha}{2\mu}\beta^{k+1},\frac{\varepsilon^k}{2}\right\}, \quad \mathrm{dist}(\tilde H'_x(x^{k+1},\varepsilon^{k+1}),\ \partial_CH(x^{k+1})) \le \gamma\beta^{k+1}.
      \end{equation*}
      If \eqref{neq: decrease} is not satisfied, set $\beta^{k+1}=\beta^k$ and $\varepsilon^{k+1}=\varepsilon^k$.
  
      (4) Set $k=k+1$ and return to step (1).
    \end{alg}

  \subsection{Structure-Preserving Smoothing Criteria and Convergence}
    The convergence analysis of Algorithm~\ref{alg: jacobian_smoothing}
    requires smoothing approximations that preserve the essential monotone structure
    while remaining consistent with the generalized Jacobian of the original
    piecewise smooth mapping.
    We therefore introduce the following class of admissible
    structure-preserving smoothing approximations.
    \begin{definition}\label{def:admissible}
      We say that the smoothing $\tilde{F}(x,\varepsilon)$ is structure-preserving admissible for \eqref{eq:solver_model} if each component $\tilde{F}_i(x,\varepsilon) = \tilde \phi_i(x_i,\varepsilon)$  satisfies the following conditions:
      \begin{enumerate}
        \item (Smoothness) $\tilde \phi_i(\cdot,\varepsilon) \in C^1\left(\mathbb{R}\right)$ for all  $\varepsilon > 0$;
        \item (Approximation) There exists $\mu>0$, such that for all $ x_i\in\mathbb{R}$ and $\varepsilon > 0$,
        \begin{equation*}
          \left|\tilde \phi_i(x_i,\varepsilon)-\phi_i(x_i)\right| \le \mu\varepsilon;
        \end{equation*}
        \item (Monotonicity preservation) $ \tilde \phi_i(\cdot,\varepsilon)$ is increasing for all  $\varepsilon > 0$;
        \item (Intermediate slope property) 
        For each $x_i\in\mathbb{R}$, there exists $\varepsilon_0>0$ such that for all $0 <\varepsilon < \varepsilon_0$, there exist $\delta_1, \delta_2 \geq 0$ satisfying 
        \begin{equation*}
            \frac{\partial \tilde \phi_i}{\partial x_i}(x_i,\varepsilon)  \in \mathrm{co}\left(\left\{\frac{\partial \phi_i}{\partial x_i}(x_i-\delta_1), \frac{\partial \phi_i}{\partial x_i}(x_i +\delta_2)\right\}\right),
        \end{equation*}
        where $\operatorname{co}(\cdot)$ denotes the convex hull. Moreover, $\delta_1,\delta_2 \rightarrow 0^+$ as $\varepsilon \rightarrow 0^+$.
      \end{enumerate}
    \end{definition}


    \begin{lemma}\label{lem: level_set}
      For any $x_0\in\mathbb{R}^n$, $\alpha\in(0,1)$, the level set
      \begin{equation*}
        D_0=\left\{x\in\mathbb{R}^n:\ \theta(x)\le(1+\alpha)^2\theta(x^0)\right\}
      \end{equation*}
      is bounded.
    \end{lemma}
    \begin{proof}
      It suffices to show that $\|H(x)\| \to \infty$ as $\|x\| \to \infty$.  
      By the strong monotonicity \eqref{eq:solver_strong_monotone},
      \[
        (x, H(x)) = (x-0, H(x) - H(0)) + (x, H(0))
        \ge \lambda_{\min} \|x\|^2 - \|x\|(\|F(0)\| + \|f\|).
      \]
      Hence $\|H(x)\| \ge \lambda_{\min}\|x\| - \|F(0)\| - \|f\| \to \infty$ as $\|x\|\to\infty$.
    \end{proof}

    \begin{lemma}\label{lem: smooth_nonsingular}
      If each component of $\tilde{F}(x,\varepsilon)$ is continuously differentiable and satisfies monotonicity preservation, then for any $\varepsilon>0$, $\tilde H'_x(x,\varepsilon)$ is symmetric positive definite and therefore nonsingular.
    \end{lemma}
    \begin{proof}
      A direct computation gives diagonal $\tilde F'_x(x,\varepsilon) = \sum_{i=1}^n \frac{\partial \tilde \phi_i}{\partial x_i}(x_i,\varepsilon) e_i e_i^T$.   
      Monotonicity preservation yields $\frac{\partial \tilde \phi_i}{\partial x_i}(x_i,\varepsilon) \ge 0$ for all $i$, implying that $\tilde F'_x(x,\varepsilon)$ is positive semidefinite.   
      Since $A$ is symmetric and positive-definite, the matrix
      \[
        \tilde H'_x(x,\varepsilon) = A + \tilde F'_x(x,\varepsilon)
      \]
      is also symmetric and positive-definite and thus nonsingular.
    \end{proof}

    \begin{lemma}\label{lem: Jacobian consistency}
      If each component of $\tilde{F}(x,\varepsilon)$ is continuously differentiable and satisfies intermediate slope property, 
      then for any $x\in\mathbb{R}^n$, 
      \begin{equation*}
        \lim_{\varepsilon\rightarrow 0^+} \mathrm{dist}(\tilde H'_x(x,\varepsilon),\partial_C H(x)) = 0.
      \end{equation*}
    \end{lemma}
    \begin{proof}
      It suffices to show that for each $i$ and any $x_i \in\mathbb{R}$,
      \begin{equation*}
        \lim_{\varepsilon\rightarrow 0^+} \mathrm{dist}\left(\frac{\partial \tilde \phi_i}{\partial x_i}(x_i,\varepsilon)  ,\partial \phi_i(x_i)\right) = 0.
      \end{equation*}
      By the intermediate slope property, for sufficiently small \(\varepsilon>0\), there exists a scalar \(\lambda\in[0,1]\) such that
      \begin{equation*}
        \frac{\partial \tilde \phi_i}{\partial x_i}(x_i,\varepsilon) = \lambda\phi_i'(x_i-\delta_1) + (1-\lambda)\phi_i'(x_i+\delta_2).
      \end{equation*} 
      From~\eqref{eq:solver_component_Bdiff}, we also have
      $
        \partial \phi_i(x_i)
        = \{\alpha \phi'_{i,1}(x_i) + (1-\alpha)\phi'_{i,2}(x_i) : \alpha \in [0,1]\}.
      $
      As $\varepsilon \to 0^+$, $\delta_1,\delta_2 \to 0^+$, and
      {
      \begin{equation*}
        \begin{aligned}
        \text{dist}(\frac{\partial \tilde \phi_i}{\partial x_i}(x_i,\varepsilon),\partial \phi_i(x_i)) \leq  & \left|\lambda \left(\phi_i'(x_i - \delta_1) - \phi_{i,1}'(x_i) \right) + (1-\lambda)\left(\phi_i'(x_i+\delta_2) -\phi_{i,2}'(x_i)\right)\right| \\ 
        \leq & \left| \phi_i'(x_i - \delta_1) - \phi_{i,1}'(x_i) \right| + \left| \phi_i'(x_i+\delta_2) - \phi_{i,2}'(x_i) \right| \rightarrow 0,
        \end{aligned}
      \end{equation*}
      }
      since $\phi_i$ is piecewise continuously differentiable.
    \end{proof}

    The previous lemmas provide the key ingredients required in the abstract convergence theory of Jacobian smoothing Newton methods.
    We therefore obtain the following global and local convergence result.

    \begin{theorem}[Convergence]\label{thm:solver_convergence}
      If $\tilde F(x,\varepsilon)$ is structure-preserving admissible, then Algorithm~\ref{alg: jacobian_smoothing} is well defined for solving~\eqref{eq:solver_model}.  
      Moreover, the generated sequence $\{x^k\}$ converges globally to the unique solution $x^*$ of $H(x)=0$, and the convergence is locally Q-quadratic.
    \end{theorem}
    \begin{proof}
      By Lemma~\ref{lem: Jacobian consistency}, the admissible smoothing approximation $\tilde F(x,\varepsilon)$ satisfies the required Jacobian consistency. 
      Thus, applying Lemmas~\ref{lem: level_set} and \ref{lem: smooth_nonsingular} yields $\lim_{k\rightarrow\infty}H(x^k) = 0$ by \cite[Thm. 3.1]{chen_global_1998}. If the sequence $\{x^k\}$ does not reach $x^*$ in finitely many steps, the strong monotonicity \eqref{eq:solver_strong_monotone} implies
      \begin{equation*}
      \lambda_{\min}(A)\left\| x^k-x^*\right\| \le \left\| H(x^k) - H(x^*)\right\|,
      \end{equation*}
      and therefore $x^k \to x^*$.  
      Moreover, since $\partial_C H (x^*) = A + \partial_C F (x^*)$ is nonsingular and  $H$ is strongly semismooth, the local Q-quadratic convergence follows directly from~\cite[Thm.~3.2]{chen_global_1998}.
    \end{proof}
  
    Theorem~\ref{thm:solver_convergence} shows that the convergence theory depends only on the abstract admissibility conditions and is therefore independent of the specific smoothing realization.

  \subsection{Arc-Based Smoothing Construction and Algorithmic Realization}\label{sec:smoothing}
    To make the above framework practically implementable, we now construct an explicit smoothing approximation for the nonsmooth mapping $F$. 
    The construction performs localized regularization only in small neighborhoods of derivative discontinuities, while preserving the original function elsewhere. 
    The key idea is to approximate the adjacent smooth branches by their tangent lines and connect them through a circular arc, thereby preserving monotonicity while maintaining consistency with the local piecewise smooth structure. 

    \begin{figure}[htbp] 
      \centering
      \includegraphics[width=0.6\textwidth]{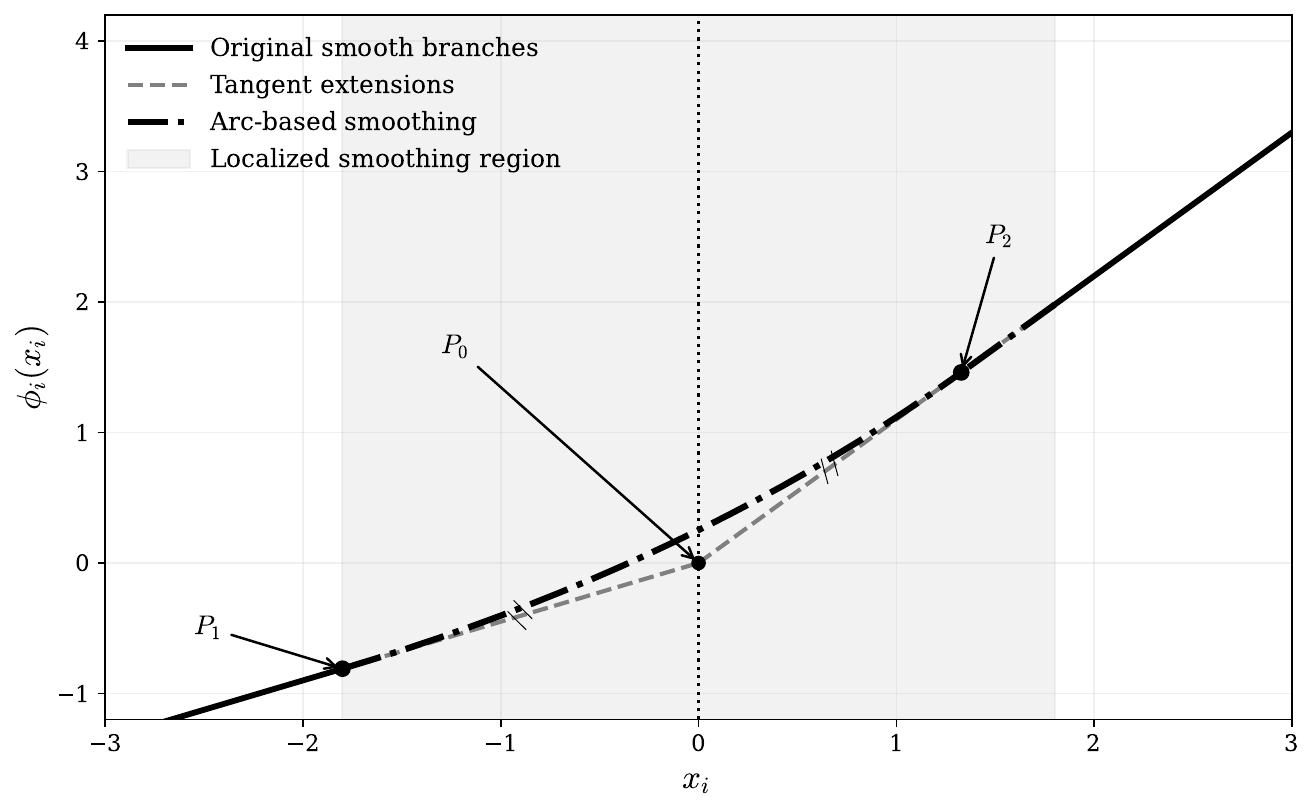}
      \caption{Illustration of the arc-based smoothing construction near a derivative discontinuity.}
      \label{fig:smoothing_arc}
    \end{figure}

    To simplify the construction, we first reduce the problem to the case of a single derivative discontinuity.
    Assume that $\phi_i(x_i)$ has discontinuous derivatives at finitely many points $\left\{a_j\right\}_{j=1}^{m_i}$ with $a_1 < a_2 < \dots < a_{m_i}$. Let $a_0=-\infty$, $a_{{m_i}+1}=+\infty$ and define $I_j = \left(a_{j-1},a_j\right)$. On each subinterval $I_j$, denote by $\phi_{i,j}:=\left.\phi_i\right|_{I_j} \in C^{2}(\bar{I}_j)$ for $1\leq j \leq {m_i}+1$. 
    Since every $\phi_{i,j}$ can be extended to a continuously differentiable function on $\mathbb{R}$ by affine continuation outside $I_j$ matching the endpoint derivatives, the original function can be decomposed as
    \begin{equation}
      \phi_i(x_i)=\sum_{j=1}^{m_i} g_{i,j}(x_i)-\sum_{j=2}^{{m_i}} \phi_{i,j}(x_i),
    \end{equation}
    where each auxiliary function $ g_{i,j}\colon \mathbb{R}\rightarrow \mathbb{R}$ possesses only a single point of derivative discontinuity, defined by
    \begin{equation}
      g_{i,j}(x_i):= \begin{cases}\phi_{i,j}(x_i), & x_i \leqslant a_j, \\ \phi_{i,j+1}(x_i), & x_i>a_j.\end{cases}
    \end{equation}
    Consequently, it suffices to construct a smooth approximation family $\tilde g_{i,j}(\cdot,\varepsilon)$ for each $g_{i,j}$, from which a smoothed version of $\phi_i$ and $F$ is obtained as
    \begin{gather}
      \tilde \phi_i(x_i,\varepsilon) :=\sum_{j=1}^{m_i} \tilde g_{i,j}(x_i,\varepsilon) - \sum_{j=2}^{m_i} \phi_{i,j}(x_i), \label{eq:solver_smoothing_arc_phi} \notag \\ 
      \tilde{{F}}( {x},\varepsilon):=
      (\tilde \phi_1(x_1,\varepsilon),\dots,\tilde \phi_n(x_n,\varepsilon))^T. \label{eq:solver_smoothing_arc_F}
    \end{gather}

    The first step relies on the tangent-extension theory established in Appendix~\ref{sec:app_tangent_extension}.
    In particular, the following proposition guarantees that, in a sufficiently small neighborhood of any derivative discontinuity, the graph of $g_{i,j}$ admits two tangent lines that either intersect or coincide inside the neighborhood.  

    \begin{proposition}[Local tangent extendability]\label{lem:locally_tangent_extension}
    Let $f\colon \mathbb{R} \to \mathbb{R}$ be a $PC^2$ function.  
    If $f'$ is discontinuous at a point $x_d$, then there exists $\varepsilon>0$ such that for every $\delta \in (0,\varepsilon)$, the function $f$ is tangent-extendable on the interval $[x_d-\delta,\,x_d+\delta]$.
    \end{proposition}
    \begin{proof}
      See Appendix~\ref{sec:app_tangent_extension}.
    \end{proof}

    Consequently, tangent-extendable neighborhoods can be identified algorithmically around derivative discontinuities. 
    Algorithm~\ref{alg:app_solver_btwindow} provides a constructive backtracking procedure for detecting such intervals, where the criterion is directly motivated by the equivalent characterization established in Proposition~\ref{prop:tangent_extension}.
    
    \begin{algorithm}
      \caption{Backtracking window detection for tangent extension}
      \label{alg:app_solver_btwindow}
      \begin{algorithmic}[1]
      \Require A $PC^2$ function $f\colon \mathbb{R} \to \mathbb{R}$, a discontinuity point $x_d$, an initial half-window width $\varepsilon^0>0$ such that $f'$ is discontinuous only at $x_d$ in $[x_d-\varepsilon^0,\,x_d+\varepsilon^0]$, and a contraction factor $\alpha\in(0,1)$.
      \Ensure A tangent-extendable half-window width $\varepsilon$.
      \State $\varepsilon \gets \varepsilon^0$
      \For{$n = 1, 2, \dots$}
          \State $d_{\max} \gets \max\!\left\{f'_+(x_d-\varepsilon),\,f'_-(x_d+\varepsilon)\right\}$
          \State $d_{\min} \gets \min\!\left\{f'_+(x_d-\varepsilon),\,f'_-(x_d+\varepsilon)\right\}$
          \If{$d_{\min} \le \frac{f(x_d+\varepsilon)-f(x_d-\varepsilon)}{2\varepsilon} \le d_{\max}$}
              \State \textbf{break}
          \EndIf
          \State $\varepsilon \gets \alpha\,\varepsilon$
      \EndFor
      \end{algorithmic}
    \end{algorithm}

    By repeatedly applying Algorithm~\ref{alg:app_solver_btwindow} to $g_{i,j}$ around $a_j$, we can obtain a strictly decreasing positive sequence $\{\varepsilon_{i,j}^k\}_{k=1}^\infty$ with $\varepsilon_{i,j}^k \to 0$ such that
    \[
    (a_j-\varepsilon_{i,j}^k,\; a_j+\varepsilon_{i,j}^k)\subset (a_{j-1},a_{j+1}).
    \]
    Consequently,  $g_{i,j}$ admits tangent extensions on each interval $(a_j - \varepsilon_{i,j}^k, a_j + \varepsilon_{i,j}^k)$. For each $k$, let ${g}^{\mathrm{tan},k}_{i,j}$ denote the associated tangent extension, which has at most one derivative discontinuity at some $x_0 \in (a_j- \varepsilon_{i,j}^k, a_j+ \varepsilon_{i,j}^k)$\footnote{If no derivative discontinuity is present, $x_0 \in (a_j- \varepsilon_{i,j}^k, a_j+ \varepsilon_{i,j}^k)$ may be chosen arbitrarily.}. By construction, it satisfies
    \begin{equation*}
      {g}^{\mathrm{tan},k}_{i,j}(a_j\pm\varepsilon_{i,j}^k) = g_{i,j}(a_j\pm\varepsilon_{i,j}^k), \quad \left({g}^{\mathrm{tan},k}_{i,j}\right)' (a_j \pm \varepsilon_{i,j}^k) = g_{i,j}'(a_j\pm \varepsilon_{i,j}^k).
    \end{equation*}
    The tangent extension is explicitly given by
    \begin{equation}\label{eq:solver_tangent_extension_g}
      {g}^{\mathrm{tan},k}_{i,j}(x):=\begin{cases}
        g_{i,j}(x), & x \in(-\infty, a_j-\varepsilon_{i,j}^k], \\ 
        g_{i,j}(a_j-\varepsilon_{i,j}^k) + g_{i,j}^{\prime}(a_j-\varepsilon_{i,j}^k)(x-a_j+\varepsilon_{i,j}^k), & x \in\left(a_j-\varepsilon_{i,j}^k, x_0\right], \\ 
        g_{i,j}(a_j+\varepsilon_{i,j}^k) + g_{i,j}^{\prime}(a_j+\varepsilon_{i,j}^k)(x-a_j-\varepsilon_{i,j}^k), & x \in\left(x_0, a_j+\varepsilon_{i,j}^k\right], \\ 
        g_{i,j}(x), & x \in(a_j+\varepsilon_{i,j}^k,+\infty).
      \end{cases}
    \end{equation}

    We now introduce a smooth geometric transition between the two affine tangent segments (see Figure~\ref{fig:smoothing_arc}).
    Specifically, let $P_0(x_0,{g}^{\mathrm{tan},k}_{i,j}(x_0))$, $P_1(x_1,{g}^{\mathrm{tan},k}_{i,j}(x_1))$ and $P_2(x_2,{g}^{\mathrm{tan},k}_{i,j}(x_2))$ be points satisfying $x_1 < x_0 < x_2$ and $|\overline{P_1P_0}|=|\overline{P_0P_2}|$. 
    Temporarily denoting ${g}^{\mathrm{tan},k}_{i,j}$ as $h$, on the interval $[x_1, x_2]$, there exists an arc with center $P_c(x_c,y_c)$ given by
    \begin{equation*}
      \left\{\begin{array}{l}
      x_c=\frac{x_2 h^{\prime}(x_1)-x_1 h^{\prime}(x_2)-\left(h(x_2)-h(x_1)\right) h^{\prime}(x_1) h^{\prime}(x_2)}{h^{\prime}(x_1)-h^{\prime}(x_2)}, \\
      y_c=\frac{(x_1-x_2) + h(x_1) h^{\prime}(x_1)-h(x_2) h^{\prime}(x_2)}{h^{\prime}(x_1)-h^{\prime}(x_2)},
      \end{array}\right.
    \end{equation*}
    which is tangent to ${g}^{\mathrm{tan},k}_{i,j}$ at both $P_1$ and $P_2$. The associated smoothing function is then defined as
    \begin{equation}\label{eq:solver_smoothing_arc_g}
      \begin{aligned}
      &g_{i,j}^{\mathrm{arc},k}(x)= 
        \begin{cases}
          g_{i,j}(x), & x \in(-\infty, a_j-\varepsilon_{i,j}^k], \\ 
          g_{i,j}(a_j-\varepsilon_{i,j}^k) + g_{i,j}^{\prime}(a_j-\varepsilon_{i,j}^k)(x-a_j+\varepsilon_{i,j}^k), & x \in\left(a_j-\varepsilon_{i,j}^k, x_1\right], \\ 
          y_c+\frac{\phi_{i,j}^{\prime}(a_j)-\phi_{i,j+1}^{\prime}(a_j)}{\left|\phi_{i,j}^{\prime}(a_j)-\phi_{i,j+1}^{\prime}(a_j)\right|} \sqrt{r^2-\left(x-x_c\right)^2}, & x \in\left(x_1, x_2\right], \\ 
          g_{i,j}(a_j+\varepsilon_{i,j}^k) + g_{i,j}^{\prime}(a_j+\varepsilon_{i,j}^k)(x-a_j-\varepsilon_{i,j}^k), & x \in\left(x_2, a_j+\varepsilon_{i,j}^k\right], \\ 
          g_{i,j}(x), & x \in(a_j+\varepsilon_{i,j}^k,+\infty),
        \end{cases}
      \end{aligned}
    \end{equation}
    where the sign of $\phi_{i,j}^{\prime}(a_j)-\phi_{i,j+1}^{\prime}(a_j)$ determines the orientation of the circular arc.
    Let $\varepsilon_{i,j}^0 = +\infty$ and for $\varepsilon \in [{\varepsilon}_{i,j}^k, {\varepsilon}_{i,j}^{k-1})$, we set $\tilde g_{i,j}(x,\varepsilon)= g_{i,j}^{\mathrm{arc},k}(x)$. This defines a piecewise constant smoothing family with respect to the smoothing parameter $\varepsilon$.

    The advantage of the arc-based smoothing lies in its ability to preserve monotonicity while ensuring a monotone transition of the derivatives within each smoothing interval.
    As a consequence, both the monotonicity-preservation condition and the intermediate slope property are satisfied.
    The resulting approximation~\eqref{eq:solver_smoothing_arc_F} therefore satisfies all admissibility conditions except possibly the approximation property, which is established below.
    \begin{proposition}[Approximation property]\label{prop:solver_arc_approximation}
      For the arc-based smoothing~\eqref{eq:solver_smoothing_arc_F}, there exists a constant $\mu_i>0$ such that, for all $x\in\mathbb{R}$ and $\varepsilon > 0$,
      \begin{equation*}
        \left|{\tilde \phi_i(x,\varepsilon)- \phi_i(x)}\right|\le \mu_i\varepsilon.
      \end{equation*}
    \end{proposition}
    \begin{proof}
      Since $\phi_i$ is piecewise $C^2$ with derivative bounded almost everywhere, it follows that
      \begin{equation*}
        \left|g_{i,j}(x) - {g}^{\mathrm{tan},k}_{i,j}(x)\right| \le 6 L_i\varepsilon_{i,j}^k,
      \end{equation*} 
      \begin{equation*}
        \left|g_{i,j}^{\mathrm{arc},k}(x) - {g}^{\mathrm{tan},k}_{i,j}(x)\right| 
        \le \left|\phi_i(a_j-\varepsilon_{i,j}^k)-\phi_i(a_j+\varepsilon_{i,j}^k)\right|
        \le 2L_i \varepsilon_{i,j}^k,
      \end{equation*}
      where $L_i$ is the Lipschitz constant of $\phi_i$.
      For $\varepsilon \in [ \varepsilon_{i,j}^k,   \varepsilon_{i,j}^{k-1})$, we have $\tilde g_{i,j}(x,\varepsilon) = g_{i,j}^{\mathrm{arc},k}(x) $. 
      Combining the above inequalities yields
      \begin{equation*}
        |\tilde g_{i,j}(x) - g_{i,j}(x)|\le 8L_i\varepsilon,
      \end{equation*}
      and
      \(
          \left|{\tilde \phi_i(x,\varepsilon)- \phi_i(x)}\right|
          \le \sum_{j=1}^{m_i} |\tilde g_{i,j}(x,\varepsilon) - g_{i,j}(x)|
          \le \mu_i\varepsilon,
      \)
      where $\mu_i$ is independent of $\varepsilon$.
    \end{proof}
  
\section{Numerical Experiments}\label{sec:experiments}
  The objective of this section is to investigate the practical performance of the proposed Jacobian smoothing Newton method for evolving piecewise smooth systems arising from parabolic problems with hysteresis. 
  Particular emphasis is placed on the convergence behavior at individual time steps, robustness throughout transient simulations involving history-dependent nonlinearities, and scalability under mesh refinement. 
  Comparisons are performed against several existing nonlinear solvers summarized in Section~\ref{sec:app_solvers}. 
  All computations are carried out using our in-house finite element code built upon the libMesh library \citep{libMeshPaper}.  

  \subsection{Benchmark validation}

    We consider the transient eddy-current problem studied in \cite{bermudezElectromagneticComputationsPreisach2017}, governed by
    \begin{equation*}
    \begin{cases}
    \displaystyle
    \frac{\partial B}{\partial t}
    -\frac{\partial}{\partial r}
    \left(
    \frac{1}{\sigma_0 r}
    \frac{\partial (rH)}{\partial r}
    \right)
    -\frac{\partial}{\partial z}
    \left(
    \frac{1}{\sigma_0}
    \frac{\partial H}{\partial z}
    \right)
    =0
    \quad \text{in } \Omega \times (0,T),
    \\[1ex]
    B = \mu_0(H+M),
    \qquad
    \mu_0 M =  \mathcal{W}(H,\xi^0),
    \\[1ex]
    H|_{t=0}=0,
    \qquad
    H|_{\partial\Omega}
    =
    \frac{I_m}{2\pi r}\sin(2\pi ft),
    \end{cases}
    \end{equation*}
    where
    \[
    \Omega=[1,1.03]\times[0,5\times10^{-4}]~\mathrm{m}^2,
    \qquad
    T=\frac{2}{f},
    \qquad
    f=150~\mathrm{Hz},
    \]
    \[
    \sigma_0=4064777~(\Omega\cdot\mathrm{m})^{-1},
    \qquad
    \mu_0=4\pi\times10^{-7}~\mathrm{N/A}^2,
    \qquad
    I_m=775~\mathrm{A},
    \]
    and the initial memory configuration $\xi^0$ corresponds to the virgin state.

    The constitutive relation is modeled by the Preisach hysteresis operator \eqref{eq:def_preisach}, whose distribution function is given by the factorized Lorentzian form
    \begin{equation*}
    \omega(r,\sigma)
    =
    \frac{N}{2}
    \left(
    1+
    \left(
    \frac{\sigma+r-\mu}{\gamma\mu}
    \right)^2
    \right)^{-1}
    \left(
    1+
    \left(
    \frac{\sigma-r+\mu}{\gamma\mu}
    \right)^2
    \right)^{-1},
    \end{equation*}
    where
    \[
    N=0.080422,
    \qquad
    \gamma=0.27382,
    \qquad
    \mu=91.24317.
    \]
    The Preisach operator is evaluated from its stored memory
    curve using an Everett-function representation;
    see \cite[Sections~4.2--4.3]{bermudez_mathematical_2014}
    and \cite{verdiNumericalApproximationPreisach1989,brokate1990some}.
    The memory curve determines the locations of memory-induced
    derivative discontinuities, while the Everett representation
    provides the one-sided derivatives required for arc-based smoothing.
    A $C^1$ approximation of the Everett function avoids
    interpolation-induced derivative jumps.
    The implementation of the Preisach model is validated against the benchmark excitation proposed in \cite{benabou_comparison_2003}; see Figure~\ref{fig:exp_preisach}.

    The spatial discretization is performed on the mesh shown in Figure~\ref{fig:exp_mesh}, and the time step size is chosen as \(\tau=T/80\).  
    To further validate the full transient simulation, Figure~\ref{fig:skin_effect} depicts the evolution of \(\mu_0M(t)\) at two representative points located at the surface and in the middle of the material sheet. 
    The computed results reproduce the skin-effect behavior and are in excellent agreement with \cite[Fig.~7]{bermudezElectromagneticComputationsPreisach2017}. 

    \begin{figure}[!htbp]
      \centering
      \includegraphics[width=0.7\textwidth]{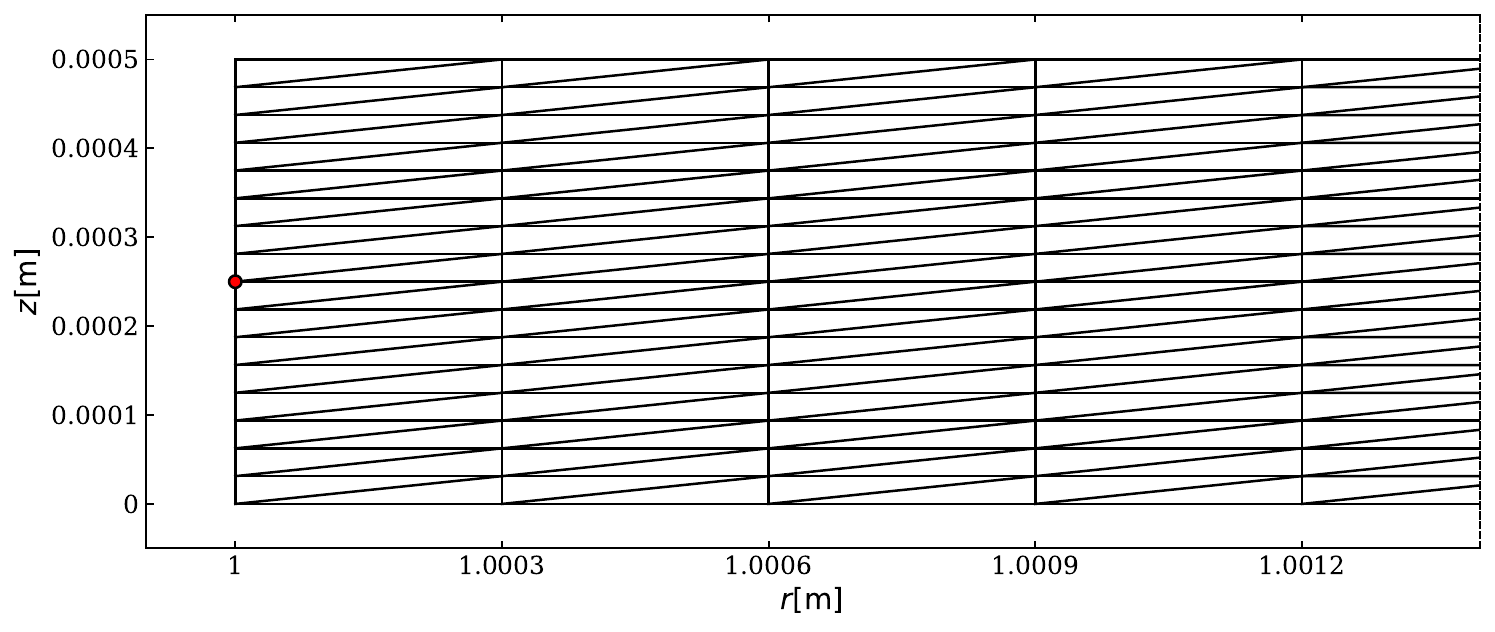}
      \caption{Detail of the finite element mesh used for the transient eddy-current simulations. The red markers indicate the representative points located at the surface, where the magnetization histories are monitored.}
      \label{fig:exp_mesh}
    \end{figure} 

    \begin{figure}[!htbp]
      \centering
      \includegraphics[width=0.5\textwidth]{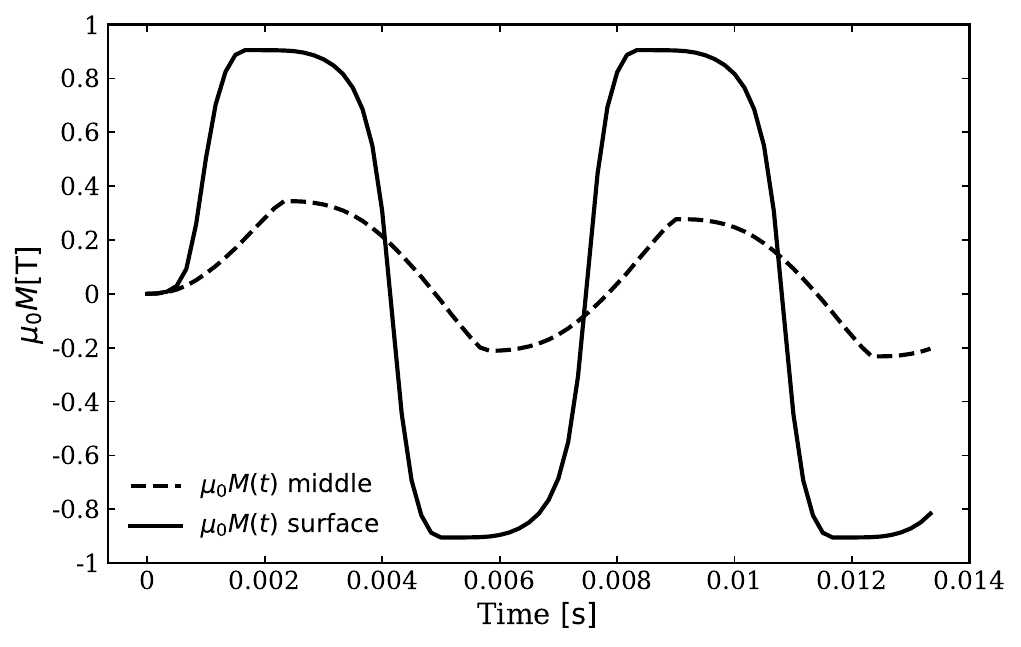}
      \caption{Time evolution of $\mu_0 M(t)$ at representative points located at the surface and in the middle of the domain.}
      \label{fig:skin_effect}
    \end{figure}

  \subsection{Nonlinear robustness and convergence}\label{subsec:rob_conv} 
    A feature of the fixed-point and duality iterations is that the system matrix can be assembled and factorized only once before the transient time-stepping loop. 
    Accordingly, the same parameter configuration is used throughout the entire simulation. 
    All linear systems are solved by the sparse direct solver MUMPS through PETSc using its default configuration. 
    Due to the initial virgin state of the hysteresis operator, the parameter choice $\beta=0$ provides a natural initialization for both the fixed-point and duality iterations. 
    For the duality iteration, the parameter $\lambda=10^3$ is selected after preliminary tuning. 
    For the semismooth Newton method, the generalized Jacobian is assembled using the right-sided derivative of each component function.
    For the proposed Jacobian smoothing Newton method (Algorithm~\ref{alg: jacobian_smoothing}), the parameters are chosen as
    \[
    \rho=\alpha=\eta=10^{-1},
    \qquad
    \gamma=10,
    \qquad
    \sigma=10^{-4},
    \qquad
    \mu=10^2.
    \]
    The initial guess at each time step is taken as the converged solution from the previous time step. 
    Although the resulting nonlinear algebraic systems share the same abstract form, the associated piecewise smooth nonlinear mappings evolve throughout the transient simulation due to the history dependence of the hysteresis operator.

    \begin{table}[!htbp]
      \centering
      \caption{Robustness and convergence statistics of nonlinear solvers over all time steps.}
      \label{tab:robustness}
      \begin{tabular}{lcccc}
      \toprule
      Solver & Status & Avg iter & Std dev & Failure mode \\
      \midrule
      Nonlinear Gauss--Seidel 
        & Converged & 27.6 & 18.0 & Slow convergence \\
      
      Fixed-point ($\beta = 0$) 
        & Failed & -- & -- & Stagnation \\
      
      Duality ($\beta = 0$, $\lambda = 10^3$) 
        & Converged & 22.4 & 21.8 & Parameter-sensitive \\
      
      Semismooth Newton 
        & Failed & -- & -- & Limit cycle oscillation \\
      
      \textbf{Jacobian Smoothing Newton} 
        & \textbf{Converged} & \textbf{2.3} & \textbf{0.57} & \textbf{Fast and stable} \\
      \bottomrule
      \end{tabular}
    \end{table}

    Table~\ref{tab:robustness} summarizes the robustness and convergence behavior of the considered nonlinear solvers over the entire transient simulation, using the stopping criterion
    $\left\|{r_k}\right\|_2 < 10^{-9}$.
    Here, ``Avg iter'' denotes the average number of nonlinear iterations per time step, while ``Std dev'' represents the corresponding standard deviation over the whole transient simulation.
    The proposed Jacobian smoothing Newton method successfully converges for all time steps while maintaining both extremely low iteration counts and very small fluctuations. 
    In contrast, the fixed-point iteration fails due to stagnation, whereas the semismooth Newton method exhibits persistent oscillatory behavior and does not converge. 
    The nonlinear Gauss--Seidel and duality methods are able to converge globally, but require substantially more iterations and exhibit significant variability throughout the transient simulation.

    \begin{figure}[!htbp]
      \centering
      \includegraphics[width=0.45\textwidth]{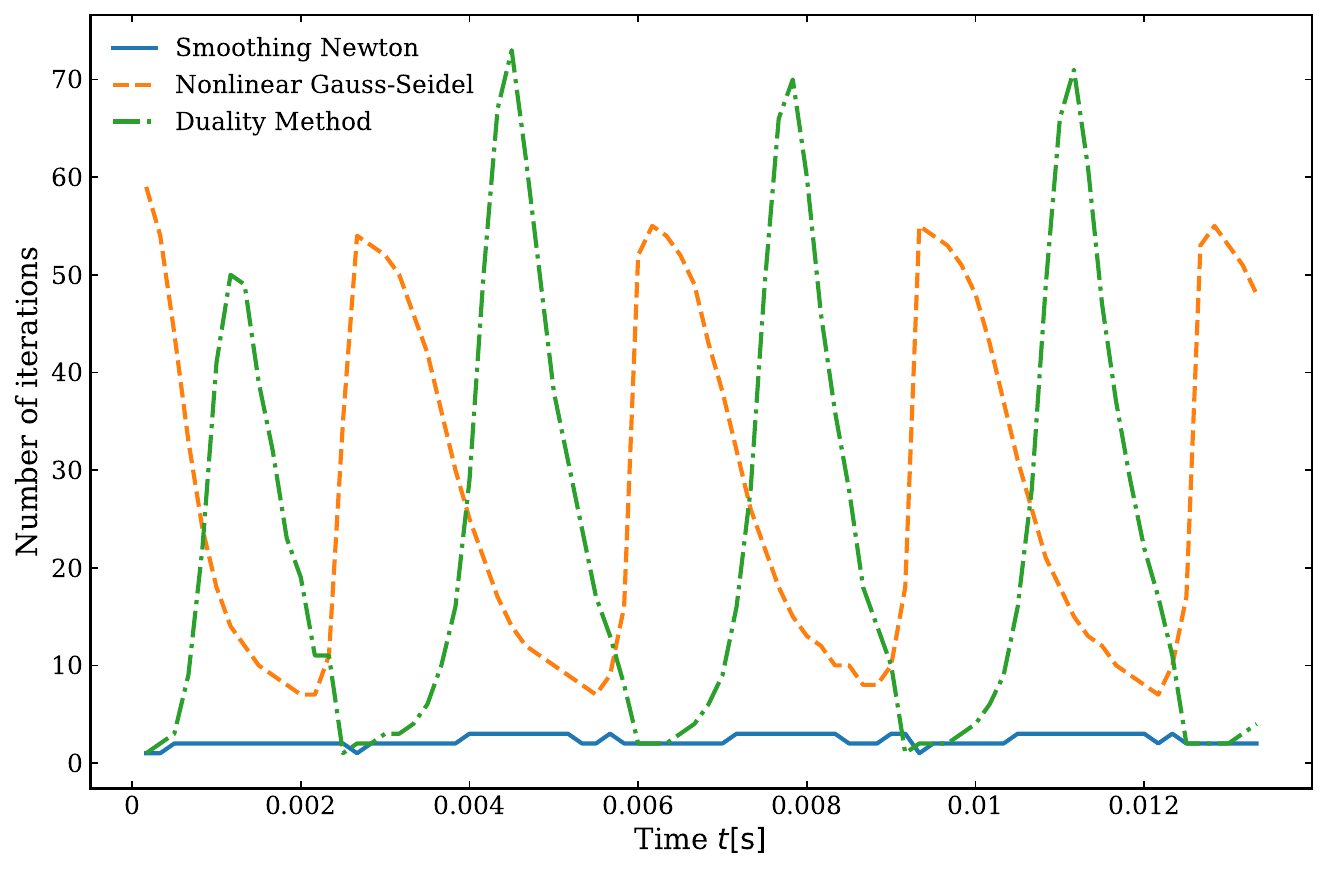}
      \put(0,10){{(a)}} 
      \qquad 
      \includegraphics[width=0.45\textwidth]{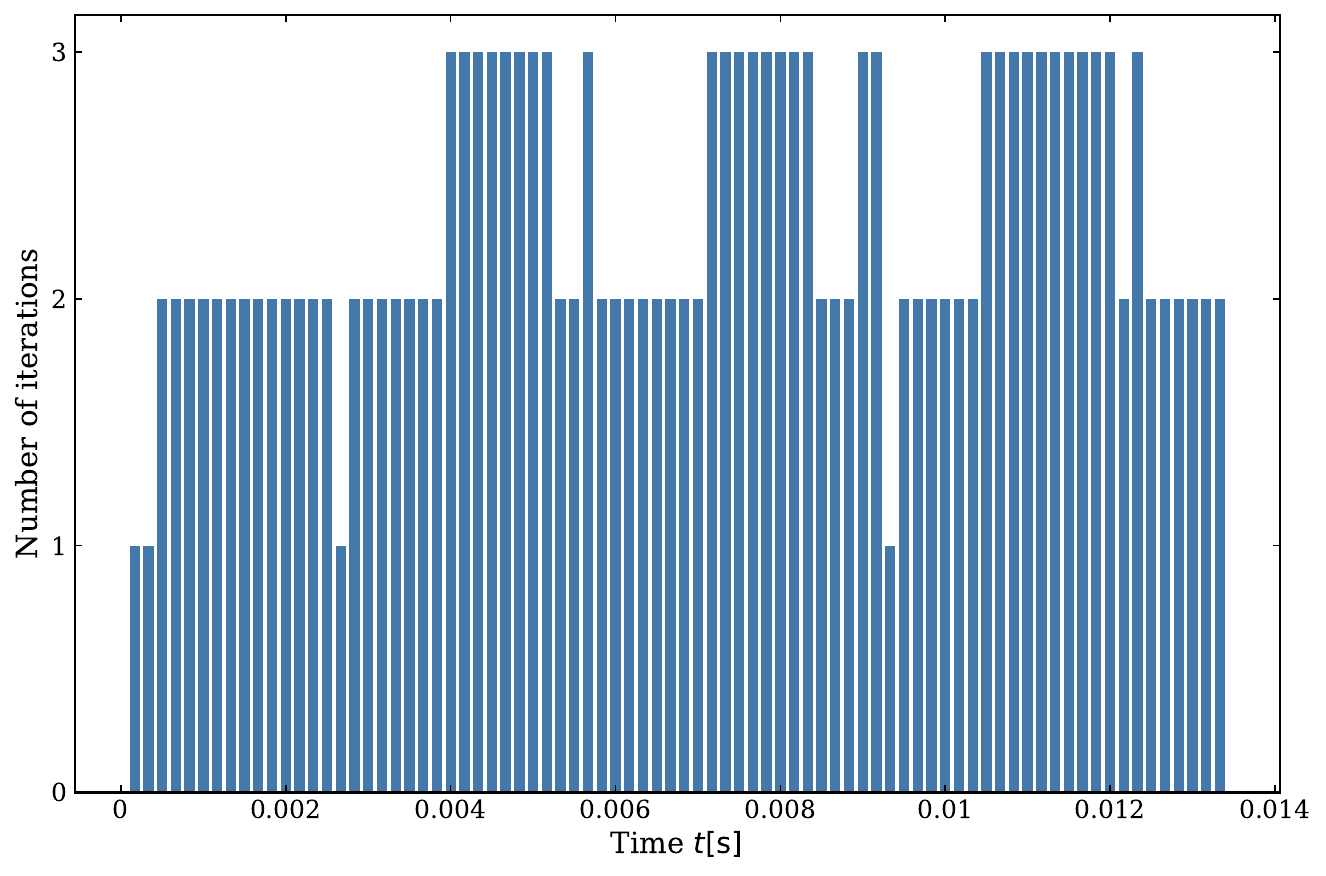}
      \put(0,10){{(b)}}
      \caption{
        Nonlinear iteration counts throughout the transient simulation.
        (a) Iteration counts per time step.
        (b) Detailed iteration counts of the Jacobian smoothing Newton method.
        }
      \label{fig:solver_iterations_times}
    \end{figure}

    \begin{figure}[!htbp]
      \centering
      \includegraphics[width=0.45\textwidth]{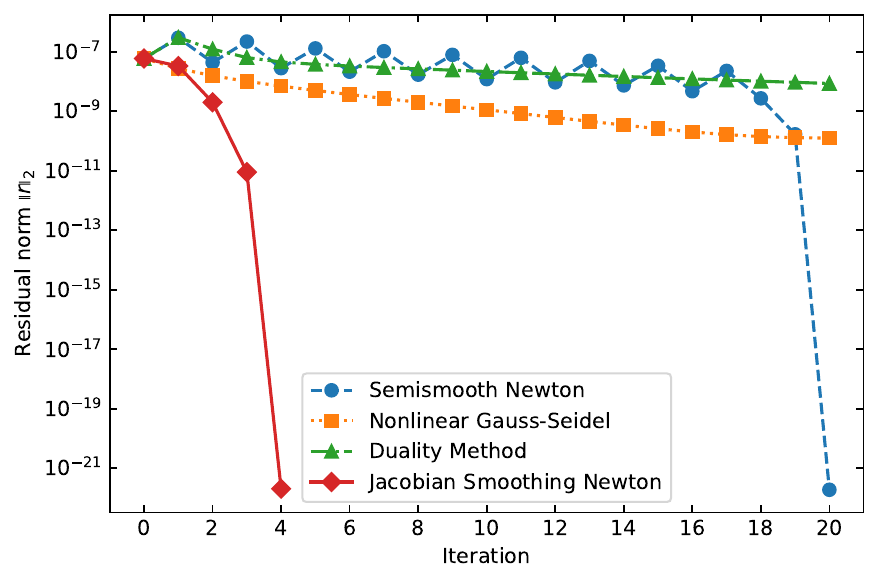}
      \put(0,10){{(a)}}
      \qquad
      \includegraphics[width=0.45\textwidth]{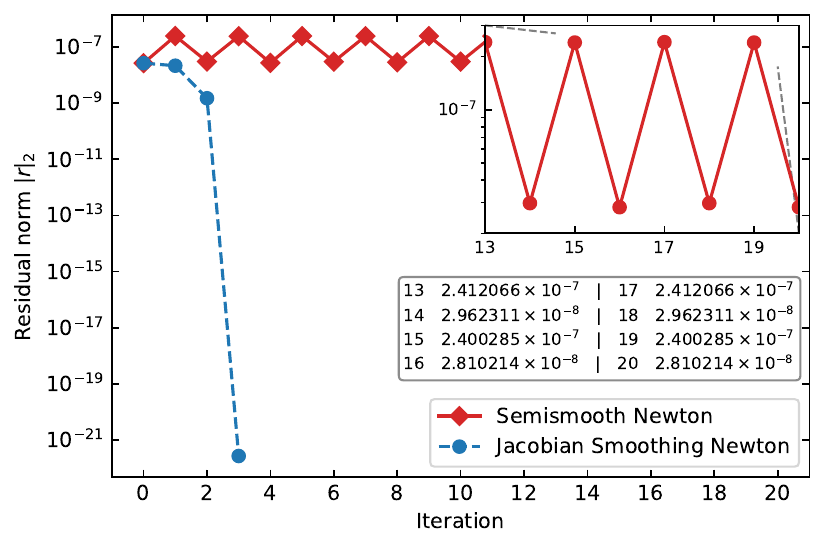}
      \put(0,10){{(b)}}
      \caption{
      Residual histories for representative single-step problems.
      (a) Convergence behavior at $t_{29}$.
      (b) Convergence behavior at $t_{31}$.
      The unit-step semismooth Newton method converges after initial oscillations at $t_{29}$ but fails to converge at $t_{31}$; the inset in panel (b) shows iterations $k=13,\dots,20$, illustrating a stable four-cycle.}
      \label{fig:solver_convergence}
    \end{figure}
    
    Figure~\ref{fig:solver_iterations_times}(a) shows the nonlinear iteration counts at each time step over the entire transient simulation. 
    Both the nonlinear Gauss--Seidel and duality methods exhibit substantial variations in iteration counts. However, their temporal patterns differ significantly. The duality iteration requires considerably more iterations near the input turning points, where the monotonicity of the hysteresis input changes, whereas the nonlinear Gauss--Seidel method exhibits an almost complementary behavior. 
    This observation indicates that different fixed-pattern iterative schemes interact with the evolving nonsmooth hysteresis structure in fundamentally different ways. Such behaviors indicate a strong sensitivity to the evolving local structure of the nonlinear mappings.
    Figure~\ref{fig:solver_iterations_times}(b) further shows that the proposed Jacobian smoothing Newton method maintains nearly uniform iteration counts throughout the transient simulation. 
    The combination of smoothing regularization and local Jacobian information enables the proposed method to adapt robustly to the changing piecewise smooth structure, thereby achieving both stable and efficient convergence behavior across all time steps.

    To investigate the convergence behavior at individual time levels, we consider representative single-step problems at $t_{29}$ and $t_{31}$, respectively, as shown in Figure~\ref{fig:solver_convergence}. 
    For these residual histories, the stopping criterion is tightened to $\|r_k\|_2<10^{-12}$ to better resolve the local convergence behavior.
    The proposed Jacobian smoothing Newton method exhibits a clear locally quadratic convergence rate. 
    In contrast, the nonlinear Gauss--Seidel and duality methods converge substantially more slowly. 
    Oscillatory behavior of the semismooth Newton iteration can be observed in both panels of Figure~\ref{fig:solver_convergence}. 
    For the problem at $t_{29}$, the iteration experiences several oscillations before eventually entering the local quadratic convergence regime. 
    In contrast, for the problem at $t_{31}$, the iteration becomes trapped in a stable four-cycle and fails to converge. 
    These results indicate that the nonsmooth switching structure induced by hysteresis may significantly affect the dynamics of semismooth Newton iterations. The smoothing regularization effectively suppresses such oscillations, leading to stable and robust convergence
    of the proposed method.

    \begin{remark}
      In the present benchmark, all trial steps of the Jacobian smoothing Newton method satisfy the line-search condition with unit step length. Backtracking is nevertheless activated for a different choice of spatial and temporal grids, where it avoids an initial residual increase and reduces the number of Newton iterations. Equipping the semismooth Newton method with Armijo backtracking also eliminates the observed cycles and yields convergence at all time steps; however, these numerical observations do not establish a general convergence guarantee for this globalized semismooth Newton method in the present setting.
    \end{remark}

  \subsection{Mesh refinement scalability}

    To further investigate the scalability of the proposed solver, we perform a mesh refinement study for the full transient problem. 
    A sequence of uniformly refined meshes based on Figure~\ref{fig:exp_mesh} is considered. 
    Table~\ref{tab:mesh_refinement} reports the nonlinear iteration statistics and total runtime for the nonlinear Gauss--Seidel, duality iteration, and proposed Jacobian smoothing Newton methods. 
    The stopping criterion is chosen as $\|r_k\|_2 < 10^{-11}$,
    which yields visually indistinguishable transient solutions, i.e., the magnetization profiles shown in Figure~\ref{fig:skin_effect}.
  
    \begin{table}[!htbp]
      \centering
      \caption{Nonlinear iteration statistics under mesh refinement.}
      \label{tab:mesh_refinement}
      \begin{tabular}{c|cc|cc|cc}
      \toprule
      & \multicolumn{2}{c|}{Smoothing Newton}
      & \multicolumn{2}{c|}{Duality Iteration}
      & \multicolumn{2}{c}{Nonlinear Gauss--Seidel} \\
      \cmidrule(lr){2-3}\cmidrule(lr){4-5}\cmidrule(lr){6-7}
      DOFs
      & Avg $\pm$ Std
      & Time [s]
      & Avg $\pm$ Std
      & Time [s]
      & Avg $\pm$ Std
      & Time [s]
      \\
      \midrule
      
      6,633
      & $3.10 \pm 0.61$
      & 5.55
      & $66.59 \pm 54.59$
      & 26.92
      & $263.71 \pm 187.71$
      & 125.46
      \\
      
      26,065
      & $3.05 \pm 0.63$
      & 23.00
      & $62.03 \pm 53.70$
      & 99.05
      & $932.88 \pm 668.52$
      & 1729.76
      \\
      
      103,329
      & $2.84 \pm 0.65$
      & 100.89
      & $54.53 \pm 48.80$
      & 352.67
      & $3268.1 \pm 2339.6$ & 24317.6
      \\
      
      411,457
      & $2.66 \pm 0.53$
      & 549.22
      & $47.01 \pm 43.44$
      & 1586.18
      & -- & --
      \\
      
      1,642,113
      & $2.56 \pm 0.50$
      & 2411.43
      & $39.64 \pm 38.35$
      & 4782.60
      & -- & --
      \\
      
      \bottomrule
      \end{tabular}
    \end{table}
    As the mesh is refined, the proposed Jacobian smoothing Newton method exhibits remarkably stable nonlinear convergence behavior. 
    The average number of nonlinear iterations remains nearly constant, while the standard deviation stays uniformly small across all refinement levels. 
    This indicates that the nonlinear iteration complexity of the proposed method is largely insensitive to both the spatial discretization level and the evolving hysteresis history.
    The duality method exhibits a noticeable reduction in both the average iteration count and the corresponding standard deviation under mesh refinement. 
    However, it still requires one order of magnitude more nonlinear iterations and exhibits substantially larger fluctuations across time steps.
    In contrast, the nonlinear Gauss--Seidel method shows severe deterioration under mesh refinement. For finer meshes, the computational cost becomes prohibitively large and the corresponding results are therefore omitted. 
    Since all inner linear systems are solved accurately by sparse direct methods, the observed iteration behavior mainly reflects intrinsic properties of the nonlinear solvers rather than inaccuracies introduced by iterative linear solves.

    \begin{figure}[!htbp]
      \centering
      \includegraphics[width=0.5\textwidth]{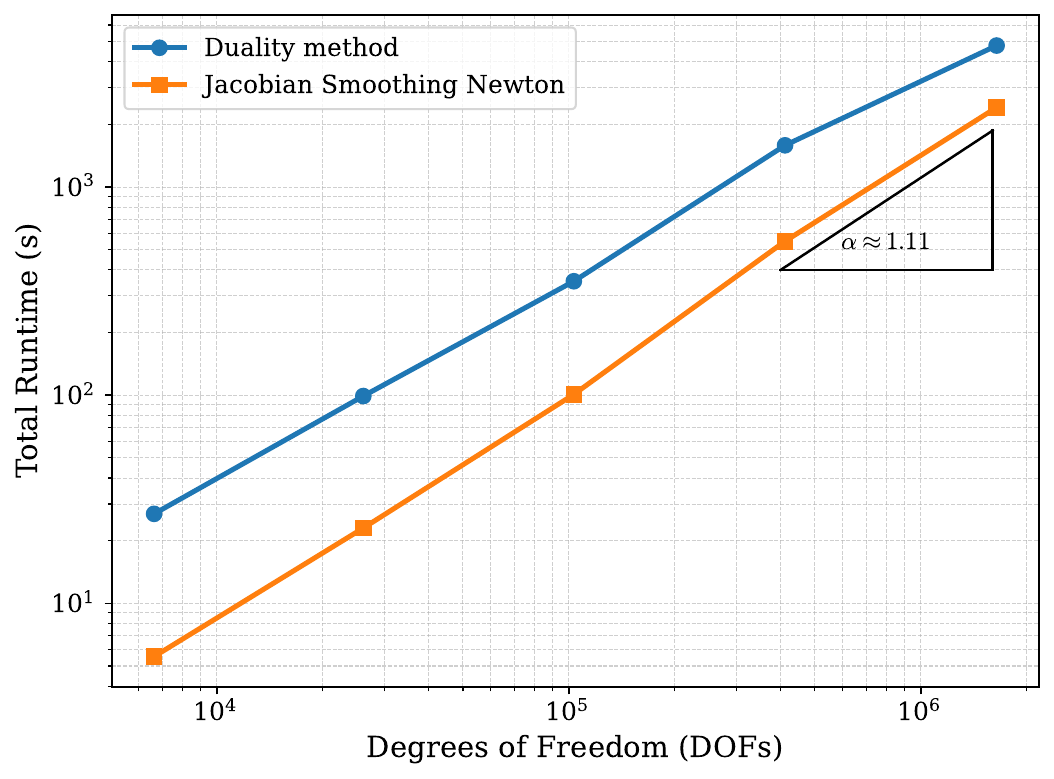} 
      \caption{
        Total runtime versus degrees of freedom under uniform mesh refinement.
        }
      \label{fig:exp_runtime_scaling}
    \end{figure} 

    The total runtime results further demonstrate the efficiency advantage of the proposed method. 
    Although each Jacobian smoothing Newton iteration involves Jacobian assembly and sparse matrix factorization, the significantly reduced number of nonlinear iterations leads to substantially lower overall computational cost.
    Figure~\ref{fig:exp_runtime_scaling} shows the total runtime versus the number of degrees of freedom under uniform mesh refinement. 
    The empirical runtime scaling remains close to linear on the log--log scale for both methods. 
    For the proposed Jacobian smoothing Newton method, the overall computational complexity is primarily governed by the underlying sparse direct linear solves, while the nonlinear iteration complexity remains essentially mesh-independent.
    Moreover, the runtime advantage of the proposed method demonstrates its practical suitability for large-scale transient simulations.
  
    Overall, the proposed Jacobian smoothing Newton method demonstrates superior robustness and efficiency for transient eddy-current problems with hysteresis, while maintaining stable nonlinear convergence behavior under mesh refinement.


\section{Conclusions}\label{sec:conclusion}

In this work, we investigated the nonlinear algebraic structures arising from fully discrete finite element approximations of parabolic hysteresis problems. 
Using a level-function formalization, we showed that the resulting nonlinear systems possess an evolving monotone piecewise smooth structure: each system is strongly monotone and piecewise smooth, while the locations of switching points and smooth branches evolve throughout the transient process due to history dependence.

Based on this characterization, we developed a structure-preserving smoothing Newton method equipped with backtracking line-search globalization. 
An explicit arc-based smoothing construction was developed and rigorously justified, and was shown to satisfy the admissibility criteria for global convergence and local Q-quadratic convergence. 
The resulting method preserves the intrinsic monotonicity while ensuring a monotone transition of the derivatives within each smoothing interval.
Numerical experiments demonstrate that the proposed method provides a robust and efficient solver for hysteresis-driven finite element systems across different stages of the transient evolution.

Future work includes extensions to more general coupled or vector-valued hysteresis operators. 
Moreover, it remains of interest to investigate whether the proposed admissibility criteria can accommodate broader classes of smoothing constructions beyond the arc-based approach considered here. 
Another important direction is the development of rigorous globalization theories for semismooth Newton methods, including line-search strategies for evolving piecewise smooth nonlinear systems.

\backmatter

\bmhead{Acknowledgements}
  The authors would like to thank the reviewers and editors for their valuable suggestions that helped to improve the paper.
  The computations were done on the high performance computers of State Key Laboratory of Scientific and Engineering Computing, Chinese Academy of Sciences.

\bmhead{Author Contributions} 
  Shu Xu: investigation, conceptualization, methodology, software, formal analysis, validation, visualization, writing and editing. Liqun Cao: funding acquisition, supervision and review. 

\bmhead{Funding} 
  This work was supported by the National Natural Science Foundation of China (Grant No.~12371437) and the Beijing Natural Science Foundation (Grant No.~Z240001).

  \bmhead{Data Availability}
  The numerical codes used in this study are available from the corresponding author upon reasonable request.

\section*{Declarations}

\bmhead{Conflict of Interest} 
  The author declares no conflict of interest.

\begin{appendices}
    
  \section{Tangent Extension}\label{sec:app_tangent_extension}
    \begin{definition}[Tangent extension]\label{def:app_tangent_extension}
      Let $f\colon D \subset \mathbb{R} \to \mathbb{R}$, where $D$ is an open set containing two distinct points $a$, $b$ with $a<b$. Suppose that the right derivative $f'_+(a)$ and the left derivative $f'_-(b)$ exist.  
      If the tangents to $f$ at $x=a$ and $x=b$ either coincide or intersect at some point $x_0 \in [a,b]$, then $f$ is said to be tangent extendable on the closed interval $[a,b]$.  
      Define the local tangent lines $f_a(x)=f'_+(a)(x-a)+f(a)$ and $f_b(x)=f'_-(b)(x-b)+f(b)$. Then the function
      \begin{equation*}
      \bar{f}(x):=\begin{cases}
      f_a(x),\ &\text{if}\ x\in [a,x_0], \\
      f_b(x),\ &\text{if}\ x\in [x_0,b], \\
      f(x),\   &else, \\
      \end{cases}
      \end{equation*}
      is called the tangent extension of $f$ on $D\cup[a,b]$.
    \end{definition}
      
    \begin{proposition}\label{prop:tangent_extension}
      Let $f\colon D \subset \mathbb{R} \to \mathbb{R}$ be defined on an open set containing $a$, $b$ with $a<b$.  
      Then $f$ is tangent extendable on $[a,b]$ if and only if
      \begin{equation}\label{neq:tangent_intersection}
      \min\left\{f'_+(a),f'_-(b)\right\} \le \frac{f(b)-f(a)}{b-a} \le \max\left\{f'_+(a),f'_-(b)\right\}.
      \end{equation}
    \end{proposition}
      
    \begin{proof}
      Let the tangents at $x=a$ and $x=b$ be given by $f_a(x)=f'_+(a)(x-a)+f(a)$ and $f_b(x)=f'_-(b)(x-b)+f(b)$. If $f'_+(a)\neq f'_-(b)$, these two lines intersect at
      \begin{equation*}
        \begin{cases}
          x_0=&\frac{1}{2}(a+b)+ \frac{f(b)-f(a)-\frac{1}{2}(b-a)(f'_+(a)+f'_-(b))}{f'_+(a)-f'_-(b)},\\ y_0=&f(a) + f'_+(a)\frac{f(b)+(a-b)f'_-(b)-f(a)}{f'_+(a)-f'_ -(b)}. 
        \end{cases} 
      \end{equation*} 
      The condition $x_0\in(a,b)$ is equivalent to \eqref{neq:tangent_intersection}, proving the claim.
    \end{proof}
  
    \begin{proposition}\label{lem:app_locally_tangent_extension}
      Let $f\colon \mathbb{R} \to \mathbb{R}$ be a $PC^2$ function.  
      If $f'$ is discontinuous at a point $x_d$, then there exists $\varepsilon>0$ such that for every $\delta \in (0,\varepsilon)$, the function $f$ is tangent extendable on the interval $[x_d-\delta,\,x_d+\delta]$.
    \end{proposition}
      
    \begin{proof}
      Without loss of generality, assume $f'_-(x_d)>f'_+(x_d)$ and denote $d=f'_-(x_d)-f'_+(x_d)$. 
      Since $f\in PC^{2}$, there exists $\varepsilon_0>0$ such that \[\left. f\right|_{[x_d-\varepsilon_0,x_d]}, \left. f\right|_{[x_d, x_d+\varepsilon_0]}  \in C^2,\] and for all $\delta_1, \delta_2, \delta_3 \in (0,\varepsilon_0)$,
      \begin{gather}
        f'(x_d-\delta_1)-f'(x_d+\delta_2)>\frac{d}{3}, \label{eq:solver_discontinuity_gap} \\ 
        f(x_d-\delta_3)-f(x_d)+\delta_3 f'(x_d-\delta_3) \ge -\frac{1}{2}L\delta_3^2, \label{eq:solver_tangent_error_lb}\\
        f(x_d+\delta_3)-f(x_d)-\delta_3 f'(x_d+\delta_3) \ge -\frac{1}{2}L\delta_3^2, \label{eq:solver_tangent_error_ub}
      \end{gather}
      where $L$ is the Lipschitz constant of $f'$ on $(x_d-\varepsilon_0,\,x_d+\varepsilon_0)$.
      Set $\varepsilon = \min\!\left\{\varepsilon_0, \tfrac{2}{3}\tfrac{d}{L}\right\}$.  
      For any $\delta \in (0,\varepsilon)$, the mean value theorem yields  
      \[
        f(x_d+\delta) - f(x_d) = f'(\xi_+)\delta, \quad \xi_+ \in (x_d,x_d+\delta).
      \]
      and then by \eqref{eq:solver_discontinuity_gap} and \eqref{eq:solver_tangent_error_lb} we have
      \begin{equation*}
        \begin{split}
          &f'(x_d-\delta)\delta + \left(f(x_d-\delta)- f(x_d)+\delta f'(x_d-\delta)\right) \\ 
      \ge& \left(f'(\xi_+) + \frac{d}{3}\right)\delta - \frac{1}{2}L\delta^2 \\ 
        =&f(x_d+\delta) - f(x_d) +(\frac{d}{3} - \frac{1}{2}L\delta)\delta  \\ 
      \ge& f(x_d+\delta) - f(x_d). 
        \end{split} 
      \end{equation*}
      Thus,
      \begin{equation*} 
        \frac{f(x_d+\delta)-f(x_d-\delta)}{2\delta} \le f'(x_d-\delta). 
      \end{equation*} 
      Similarly, applying  \eqref{eq:solver_discontinuity_gap}, \eqref{eq:solver_tangent_error_ub} and the mean value theorem  gives
      \begin{equation*} 
        f'(x_d+\delta)\delta - [f(x_d+\delta)- f(x_d)-\delta f'(x_d+\delta)] \le f(x_d) - f(x_d-\delta),
      \end{equation*}
      and hence 
      \begin{equation*}
        \frac{f(x_d+\delta)-f(x_d-\delta)}{2\delta} \ge f'(x_d+\delta).
      \end{equation*}
      Therefore, condition \eqref{neq:tangent_intersection} holds, and by Proposition~\ref{prop:tangent_extension}, the function $f$ is tangent extendable on $[x_d-\delta,\,x_d+\delta]$.
    \end{proof}

  \section{Existing Solvers}\label{sec:app_solvers}
    This section lists all the numerical solution methods to Problems~\ref{prob:semi_fem_full} and \ref{prob:quasi_fem_full} in the literature. Related convergence theorems and estimates for convergence factors can be found in \cite{white_parallel_1986,kojima_extension_1986,gallardo_generalized_2005,xu_multiscale_2024}.
    \subsection{Nonlinear SOR Iteration}
    Let $A = D_A - L_A - U_A$ the usual splitting of A and the nonlinear mapping $G\colon \mathbb{R}^n \rightarrow \mathbb{R}^n$, with $G_i = \left((D_A)_{ii} + F_i\right)^{-1}$. The nonlinear SOR method ($ \omega \in (0,2)$) reads 
    \begin{equation}
      u^n = (1-\omega)u^{n-1} +\omega G\left(f + L_Au^{n} + U_Au^{n-1}\right) .
    \end{equation}

    \subsection{Fixed-Point Iteration}
      To separate the linear and nonlinear parts of $F(u)$, we introduce $F^\beta(u) = F(u) - \beta D u$ for $\beta\ge0$. Then $F(u) = \beta D u + F^\beta(u)$ and the model equation \eqref{eq:solver_model} can be rewritten as
      \begin{equation}
      (A+\beta D) u + F^\beta(u) = f.
      \label{eq:app_solver_fp_split}
      \end{equation}
      This naturally leads to the following fixed-point iteration scheme:
      \begin{equation}\label{eq:app_solver_fixedpoint}
      (A+\beta D) u^{n+1} = f - F^\beta(u^n). 
      \end{equation}
      The main advantage of \eqref{eq:app_solver_fixedpoint} is that the coefficient matrix of the linear system remains constant when 
      $\beta$ is fixed. Consequently, matrix factorization or the construction of a suitable preconditioner only needs to be performed once, significantly reducing computational overhead.
  
    \subsection{Duality Iteration}
      Let $F^\beta_\lambda(u) = \frac{u - J^\beta_\lambda(u)}{\lambda}$, where $J^\beta_\lambda = (I+\lambda F^\beta)^{-1}$ denotes the resolvent of $F^\beta(u)$. Combining \eqref{eq:app_solver_fp_split} with the identity
      \begin{equation*}
      F^\beta(u) = F^\beta_\lambda(u+\lambda F^\beta(u)),
      \end{equation*}
      the following iterative scheme can be constructed for \eqref{eq:solver_model}:
      \begin{equation*}
        \begin{cases}
        (A+\beta D) u^{n+1} = f - q^{\beta,n}, \\
        q^{\beta,n+1} = F^\beta_\lambda(u^{n+1}+\lambda q^{\beta,n}),
        \end{cases}
        \label{eq:parabolic_solver_dual_iter}
      \end{equation*}
      where each iteration involves solving a linear system and $n$ nonlinear scalar equations. It can been seen that an extra correction step is introduced comparing the fixed-point iteration \eqref{eq:app_solver_fixedpoint}. 

    \subsection{Semismooth Newton Method}
    (1) Unless a stopping rule is satisfied, solve for $d_k$ with arbitrary\footnote{See \cite{kojima_extension_1986,xu_multiscale_2024} for recommended choice.} $G\left(x_k\right) \in \partial H\left(x_k\right)$,
    \begin{equation}
      G\left(x_k\right) d_k=-H\left(x_k\right).
    \end{equation}
    (2) Set $x_{k+1}:=x_k+d_k, k=k+1$, and go to (1).

\end{appendices}


\bibliography{bib/app.bib,bib/constitution.bib,bib/pde.bib,bib/ref.bib,bib/fem.bib} 

\end{document}